\documentclass[a4paper]{article}
\usepackage{amsmath,amsthm,amssymb,amscd,graphicx}
\usepackage[all]{xy}
\usepackage{amsfonts}
\usepackage{latexsym}
\usepackage{soul}
\usepackage{comment}
\usepackage{color}

\usepackage{amsmath, latexsym, amssymb}
\numberwithin{equation}{section}
\newtheorem{lemma}{Lemma}[section]
\newtheorem{proposition}[lemma]{Proposition}
\newtheorem{theorem}[lemma]{Theorem}

\newtheorem{definition}[lemma]{Definition}
\newtheorem{remark}[lemma]{Remark}

\def\address#1#2{\begingroup
\noindent\parbox[t]{12cm}{%
\small{\scshape\ignorespaces#1}\par\vskip1ex
\noindent\small{\itshape E-mail address}%
\/: #2\par\vskip4ex}\hfill%
\endgroup}%
\makeatother

\title{Second order deformations of 
associative submanifolds 
in nearly parallel $G_2$-manifolds}
\author{Kotaro Kawai
\footnote{This work was supported by JSPS KAKENHI Grant Numbers JP14J07067, JP17K14181. 
}
}
\date{}

\begin{document}

\maketitle

\begin{abstract}
Associative submanifolds $A$ in nearly parallel $G_2$-manifolds $Y$ 
are minimal 3-submanifolds in spin 7-manifolds with a real Killing spinor. 
The Riemannian cone over $Y$ has the holonomy group contained in ${\rm Spin(7)}$ 
and the Riemannian cone over $A$ is a Cayley submanifold.

Infinitesimal deformations of associative submanifolds were 
considered by the author \cite{Kdeform}. 
This paper is a continuation of the work. 
We give a necessary and sufficient condition 
for an infinitesimal associative deformation 
to be integrable (unobstructed) to second order explicitly.
As an application, we show that 
the infinitesimal deformations of a homogeneous associative submanifold 
in the 7-sphere 
given by Lotay \cite{Lotay3}, which he called $A_3$, 
are unobstructed to second order. 
\end{abstract}

\section{Introduction}

Associative submanifolds $A$ in nearly parallel $G_2$-manifolds 
are minimal 3-submanifolds in spin 7-manifolds $Y$ with a real Killing spinor. 
The Riemannian cone over $Y$ has the holonomy group contained in ${\rm Spin(7)}$ 
and the Riemannian cone over $A$ is a Cayley submanifold.
There are many examples of associative submanifolds.
For example, 
special Legendrian submanifolds 
and invariant submanifolds in the sense of \cite[Section 8.1]{Blair} 
in Sasaki-Einstein manifolds are associative. 
Lagrangian submanifolds in the sine cones of nearly K\"ahler
6-manifolds are also associative  (\cite[Propositions 3.8, 3.9 and 3.10]{Lotay3}).

We are interested in deformations of associative submanifolds in nearly parallel $G_2$-manifolds. 
Since associative deformations are equivalent to Cayley cone deformations,
it may help to develop the deformation theory of a Cayley submanifold with conical singularities. 
This study can also be regarded as an analogous study of associative submanifolds in torsion-free $G_2$-manifolds.
%

%
The standard 7-sphere $S^7$ has a natural nearly parallel $G_2$-structure. 
Lotay \cite{Lotay3} studied associative submanifolds in $S^7$ intensively. 
In particular, he classified 
homogeneous associative submanifolds 
(\cite[Theorem 1.1]{Lotay3}), 
in which he gave the first explicit homogeneous example which does not arise from other geometries.
He called it $A_3$. 
This is the only known example of this property  
up to the ${\rm Spin}(7)$-action. 
Hence 
$A_3$ is a very mysterious example. 
It would be very interesting to 
see whether it is possible to obtain other new associative submanifolds 
not arising from other geometries by deforming it.  

It is known that the expected dimension of the moduli space 
of associative submanifolds 
is 0. 
However, 
there are many examples 
which have nontrivial deformations 
as pointed out in \cite[Theorem 1.3]{Lotay3}. 
In \cite{Kdeform}, 
the author studied infinitesimal associative deformations 
of homogeneous associative submanifolds in $S^7$. 
Infinitesimal associative deformations of other homogeneous examples than $A_3$ 
are unobstructed (namely, they extend to actual deformations) 
or reduced to the Lagrangian deformation problems in a totally geodesic $S^6$
(\cite[Theorems 1.1 and 1.2]{Kdeform}).
However, 
we did not know whether 
infinitesimal associative deformations of $A_3$ are unobstructed or not 
(\cite[Theorem 1.1]{Kdeform}).
The associative submanifold $A_3$ does not arise from other known geometries 
so its deformations are more complicated.

In this paper, we study second order deformations 
of associative submanifolds. 
Second order deformations 
of other geometric objects are 
considered by many people. 
For example, see \cite{Foscolo, Koiso, Mukai}.
We give a necessary and sufficient condition 
for an infinitesimal associative deformation 
to be integrable (unobstructed) to second order explicitly 
(Lemma \ref{lem:unob second order} and Proposition \ref{prop:second derivative F}).
As an application, we obtain the following.

\begin{theorem} \label{thm:A3 unob to second order}
All of the infinitesimal deformations of the associative submanifold $A_3$ defined by 
(\ref{eq:def A3}) 
in $S^7$ 
are
unobstructed to second order. 
\end{theorem}

As stated above, 
the expected dimension of the moduli space of associative submanifolds is 0. 
Thus we will expect that an associative submanifold 
does not admit associative deformations generically. 
Theorem \ref{thm:A3 unob to second order} is unexpected result 
because it implies that 
infinitesimal associative deformations of $A_3$ might extend to actual deformations. 
(For example, by the action of some group.) 
Unfortunately, we have no idea currently.

If all infinitesimal associative deformations of $A_3$ are unobstructed, 
we will be able to know the type of singularities of Cayley submanifolds in some cases. 
Namely, as in \cite[Theorem 1.1]{Lotay_stab}, 
we can expect that 
if a Cayley integral current has a multiplicity one tangent cone 
of the form $\mathbb{R}_{>0} \times A_3$
with isolated singularity at an interior point $p$, then it has a conical singularity at $p$.
Moreover, as in \cite[Theorem 1.3]{Lotay_stab}, 
it might be useful to construct Cayley submanifolds with conical singularities 
in compact manifolds with ${\rm Spin}(7)$ holonomy. 

\begin{remark}
In \cite{Ksquashed}, the author 
classified homogeneous associative submanifolds  
and studied their associative deformations
in the squashed 7-sphere, which is a 7-sphere with another nearly parallel $G_2$-structure. 
In this case, 
all of 
homogeneous associative submanifolds 
arise from pseudoholomorphic curves of the nearly K\"ahler $\mathbb{C}P^3$. 
Thus 
the deformation problems are easier 
and all infinitesimal associative 
deformations of  homogeneous associative submanifolds 
in the squashed $S^7$ 
are unobstructed (\cite[Theorem 1.6]{Ksquashed}). 
\end{remark}

This paper is organized as follows. 
In Section 2, we review the fundamental
facts of $G_2$ and  ${\rm Spin}(7)$ geometry.
In Section 3, we recall the infinitesimal deformations of associative submanifolds 
and 
consider their second order deformations. 
We give a necessary and sufficient condition 
for an infinitesimal associative deformation 
to be integrable (unobstructed) to second order 
(Lemma \ref{lem:unob second order})
and 
describe it explicitly (Proposition \ref{prop:second derivative F}). 
In Section 4, 
we prove Theorem \ref{thm:A3 unob to second order} 
by using Proposition \ref{prop:second derivative F} 
and the Clebsch-Gordan decomposition. 
We also describe the trivial deformations 
(deformations given by the ${\rm Spin}(7)$-action) of $A_3$ explicitly.

\noindent{{\bf Notation}}:
Let $(M, g)$ be a Riemannian manifold. 
We denote by $i(\cdot)$ the interior product. 
For a tangent vector $v \in TM$, define a cotangent vector $v^{\flat} \in T^* M$ 
by $v^{\flat} = g(v, \cdot)$. 
For a cotangent vector $\alpha \in T^* M$, define a tangent vector $\alpha^{\sharp} \in TM$ 
by $\alpha = g(\alpha^{\sharp}, \cdot)$.  
For a vector bundle $E$ over $M$, 
we denote by $C^{\infty}(M, E)$ the space of all smooth sections of $E \rightarrow M$.

\noindent{{\bf Acknowledgements}}: 
The author would like to thank H\^ong V\^an L\^e 
for suggesting the problems in this paper. 
He thanks the referee for the careful reading of an earlier version of this paper 
and useful comments on it.

\section{$G_{2}$ and ${\rm Spin}(7)$ geometry}
First, we review the fundamental
facts of $G_2$ and  ${\rm Spin}(7)$ geometry.

\begin{definition} \label{def on R7}
Define a $3$-form $\varphi_{0}$ on $\mathbb{R}^{7}$ by
\begin{eqnarray*}
\varphi_{0} = dx_{123} +dx_{1} (dx_{45} +dx_{67}) +dx_{2}(dx_{46} - dx_{57}) - dx_{3}(dx_{47} + dx_{56}), 
\end{eqnarray*}
where  $(x_{1}, \cdots, x_{7})$ is the standard coordinate system 
on $\mathbb{R}^{7}$ 
and wedge signs are omitted. 
The Hodge dual of $\varphi_{0}$ is given by 
\begin{eqnarray*}
*\varphi_{0} = dx_{4567} +dx_{23} (dx_{67} + dx_{45}) +dx_{13}(dx_{57} - dx_{46}) - dx_{12}(dx_{56} + dx_{47}). 
\end{eqnarray*}

Decompose $\mathbb{R}^{8} = \mathbb{R} \oplus \mathbb{R}^{7}$
and denote by $x_{0}$ the coordinate on $\mathbb{R}$. 
Define a self-dual $4$-form $\Phi_{0}$ on $\mathbb{R}^{8}$ by
\begin{align*}
\Phi_{0} = dx_{0} \wedge \varphi_{0} + * \varphi_{0}. 
\end{align*}
Identifying $\mathbb{R}^{8} \cong \mathbb{C}^{4}$ via 
\begin{align} \label{eq:identification R8 C4}
\mathbb{R}^{8} \ni (x_{0}, \cdots, x_{7}) 
\mapsto 
(x_{0} + i x_{1}, x_{2} + i x_{3}, x_{4} + i x_{5}, x_{6} + i x_{7}) 
=: (z_{1}, z_{2}, z_{3}, z_{4}) \in \mathbb{C}^{4},
\end{align} 
$\Phi_{0}$ is described as 
\begin{align*}
\Phi_{0} = \frac{1}{2} \omega_{0} \wedge \omega_{0} + {\rm Re} \Omega_{0}, 
\end{align*}
where $\omega_{0} = \frac{i}{2} \sum_{j = 1}^{4} dz_{j \overline{j}}$ and 
$\Omega_{0} = dz_{1234}$ are the standard 
K\"{a}hler form and the holomorphic volume form on $\mathbb{C}^{4}$, respectively. 
\end{definition}

The stabilizers of $\varphi_{0}$ and $\Phi_{0}$ are 
the Lie groups $G_{2}$ and ${\rm Spin}(7)$, respectively: 
\begin{eqnarray*}
G_{2} = \{ g \in GL(7, \mathbb{R})  ;  g^{*}\varphi_{0} = \varphi_{0} \}, \qquad 
{\rm Spin}(7) = \{ g \in GL(8, \mathbb{R})  ;  g^{*}\Phi_{0} = \Phi_{0} \}. 
\end{eqnarray*}

The Lie group $G_{2}$ fixes 
the standard metric $g_{0} = \sum^7_{i=1} (dx_{i})^{2}$ and the orientation on $\mathbb{R}^{7}$. 
They are uniquely determined by $\varphi_{0}$ via 
\begin{eqnarray}\label{g varphi}
6 g_{0}(v_{1}, v_{2}) {\rm vol}_{g_{0}} = i(v_{1})\varphi_{0} \wedge i(v_{2})\varphi_{0} \wedge \varphi_{0}, 
\end{eqnarray}
where ${\rm vol}_{g_{0}}$ is a volume form of $g_{0}$ 
and $v_{i} \in T(\mathbb{R}^{7})$.

Similarly, 
${\rm Spin}(7)$ fixes 
the standard metric $h_{0} = \sum^7_{i=0} (dx_{i})^{2}$ and the orientation on $\mathbb{R}^{8}$. 
We have the following identities: 
\begin{eqnarray} \label{g Phi}
\Phi_{0} ^{2} = 14 {\rm vol}_{h_{0}}, \qquad
(i(w_{2}) i(w_{1}) \Phi_{0})^{2} \wedge \Phi_{0} = 6 
\|  w_{1} \wedge w_{2} \|_{h_{0}}^{2} 
{\rm vol}_{h_{0}}, 
\end{eqnarray}
where ${\rm vol}_{h_{0}}$ is a volume form of $h_{0}$ 
and $w_{i} \in T(\mathbb{R}^{8})$.

\begin{definition}
Let $M^7$ be an oriented 7-manifold and 
$\varphi$ be a 3-form on $M^7$. 
A 3-form $\varphi$ is called a {\bf $G_{2}$-structure} on $M^7$ if 
for each $p \in M^7$, there exists an oriented isomorphism between $T_{p} M^7$ and $\mathbb{R}^{7}$ 
identifying $\varphi_{p}$ with $\varphi_{0}$. 
From (\ref{g varphi}), $\varphi$ induces the metric $g$ 
and the volume form on $M^7$.
Similarly, for an oriented 8-manifold with a 4-form $\Phi$, 
we can define a {\bf ${\rm Spin}(7)$-structure} by $\Phi_0$.
\end{definition}

\begin{lemma}
A  $G_{2}$-structure $\varphi$ is called {\bf torsion-free} if 
$ d\varphi = d*\varphi = 0$. 
A ${\rm Spin}(7)$-structure $\Phi$ is called {\bf torsion-free} if 
$d\Phi = 0$. 
It is well-known that 
a $G_{2}$- or ${\rm Spin}(7)$-structure is torsion-free if and 
only if the holonomy group is contained in $G_{2}$ or ${\rm Spin}(7)$. 
This is also equivalent 
to saying that 
$\varphi$ or $\Phi$ is parallel with respect to 
the Levi-Civita connection of the induced metric.
\end{lemma}

\begin{definition}[{\cite[Proposition 2.3]{AleSem}}] \label{def:cha_NP}
Let $(M^7, \varphi, g)$ be a manifold with a $G_2$-structure. 
Let $\nabla$ be the Levi-Civita connection of $g$. 
A $G_2$-structure $\varphi$ is called a {\bf nearly parallel $G_{2}$-structure} if 
one of the following equivalent conditions is satisfied. 
\begin{enumerate}
\item $d \varphi = 4 * \varphi$, 
\item $\nabla \varphi = \frac{1}{4} d \varphi$,  
\item $\nabla \varphi = * \varphi$, 
\item $\nabla_{v} (* \varphi) = - v^\flat \wedge \varphi$ for any $v \in TM$, 
\item $i(v) \nabla_{v} \varphi = 0$ for any $v \in TM$, 
\item The Riemannian cone $C(M) = \mathbb{R}_{>0} \times M$ 
         admits a torsion-free {\rm Spin(7)}-structure                                           
        $\Phi = r^{3} dr \wedge \varphi + r^{4} * \varphi$ with the induced cone metric 
        $\overline{g} = dr^{2} + r^{2} g$.      
\end{enumerate}
We call a manifold with a nearly parallel $G_{2}$-structure 
a {\bf nearly parallel $G_{2}$-manifold} for short. 
\end{definition}

\begin{definition} \label{def:Crchi}
Let $(M^7, \varphi, g)$ be a manifold with a $G_2$-structure. 
Define the cross product $\cdot \times \cdot: TM \times TM \rightarrow TM$
and 
a tangent bundle valued 3-form 
$\chi \in \Omega^3 (M, TM)$ by 
\begin{align*}
g(x \times y, z) = \varphi (x,y,z), \qquad
g(\chi (x,y,z), w) = * \varphi (x,y,z,w)
\end{align*}
for $x,y,z,w \in TM$. 
They are related via 
\begin{align}\label{eq:chicr}
\chi (x,y,z) = - x \times (y \times z) - g(x,y) z + g(x,z)y. \\
\nonumber
\end{align}
\end{definition}

Next, we summarize the facts 
about submanifolds in $G_2$ and ${\rm Spin}(7)$ settings. 
Let $M^7$ be a manifold with a $G_{2}$-structure $\varphi$ and the induced metric $g$.

\begin{lemma}[\cite{Harvey Lawson}] \label{def_asso_coasso}
For 
every oriented $k$-dimensional subspace $V^{k} \subset T_{p}M^7$ 
where $p \in M^7$ and  $k = 3, 4,$ 
we have
$
\varphi|_{V^{3}} \leq {\rm vol}_{V^{3}}, \ 
*\varphi|_{V^{4}} \leq {\rm vol}_{V^{4}}.
$
An oriented 3-submanifold $L^{3} \subset M^7$ is called {\bf associative} 
if $\varphi|_{TL^{3}} = {\rm vol}_{L^{3}}$, 
which is equivalent to $\chi|_{TL^{3}} = 0$ and $\varphi|_{TL^{3}} > 0$. 
An oriented 4-submanifold $L^{4} \subset M^7$ is called {\bf coassociative} 
if $*\varphi|_{TL^{4}} = {\rm vol}_{L^{4}}$, 
which is equivalent to 
$\varphi|_{TL^{4}} = 0$ and $*\varphi|_{TL^{4}} > 0.$
\end{lemma}

Associative submanifolds have the following good properties 
with respect to the cross product. 

\begin{lemma} \label{lem:cross asso}
Let $L^3 \subset M^7$ be an associative submanifold 
and  $\nu \rightarrow L$ be 
the normal bundle of $L^3$ in $M^7$.
Then we have 
$$
TL \times TL \subset TL, \qquad
TL \times \nu \subset \nu,  \qquad
\nu \times \nu \subset TL.
$$
Here, the left hand sides are the spaces given by 
the cross product of elements of $TL$ or $\nu$. 
\end{lemma}

\begin{definition}
Let $X$ be a manifold with a ${\rm Spin}(7)$-structure $\Phi$.  
Then for 
every oriented $4$-dimensional subspace $W \subset T_{x}X$ 
where $x \in X,$ 
we have
$
\Phi|_{W} \leq {\rm vol}_{W}. 
$
An oriented 4-submanifold $N \subset X$ is called {\bf Cayley} 
if $\Phi|_{TN} = {\rm vol}_{N}$.
\end{definition}

\begin{lemma}[\cite{Harvey Lawson}]
If a $G_2$-structure is torsion-free, 
$\varphi$ and $* \varphi$ define calibrations. 
Hence 
compact (co)associative submanifolds 
are volume minimizing in their homology classes, 
and hence minimal. 
We also know that 
any (not necessarily compact) (co)associative submanifolds 
are minimal. 
Similar statement holds 
for Cayley submanifolds in a manifold with a torsion-free ${\rm Spin}(7)$-structure.
\end{lemma}

\begin{lemma}
Let $(M^7, \varphi, g)$ be a nearly parallel $G_{2}$-manifold. 
Then there are no coassociative submanifolds in $M$(\cite[Lemma 3.2]{Lotay3}). 
An oriented 3-dimensional submanifold $L \subset M$ is associative if and only if 
$C(L) = \mathbb{R}_{>0} \times L \subset \mathbb{R}_{>0} \times M = C(Y)$ is Cayley. 
In particular, $L$ is minimal. 
\end{lemma}

\section{Deformations of associative submanifolds}

\subsection{Infinitesimal deformations of associative submanifolds} \label{sec: infinitesimal}

First, we describe the infinitesimal deformation space explicitly again.  
The arguments here are based on 
\cite[Section 2]{Gayet}, 
\cite[Section 6.1]{Huybrechts}, 
\cite[Section 3.1]{MS}. 

Let $(M^7, \varphi, g)$ be a manifold with a $G_2$-structure and 
let $L^3 \subset M^7$ be a compact associative submanifold. 
Let $\nu \rightarrow L$ be 
the normal bundle of $L^3$ in $M^7$. 
By the tubular neighborhood theorem there exists a neighborhood
of $L$ in $M$ which is identified with 
an open neighborhood $\mathcal{T} \subset \nu$ of the zero section by
the exponential map. 
Set
\begin{align*}
C^{\infty}(L, \mathcal{T}) = \{ v \in C^{\infty}(L, \nu); v_{x} \in \mathcal{T} \mbox{ for any } x \in L \}.
\end{align*}
The exponential map induces the embedding 
${\rm exp}_{V}: L \hookrightarrow M$
by ${\rm exp}_{V}(x) = {\rm exp}_{x} (V_{x})$ 
for $V \in C^{\infty}(L, \mathcal{T})$ and $x \in L$. 
Let 
$$P_V: TM|_L \rightarrow TM |_{\exp_V (L)} \qquad 
\mbox{ for } V \in C^{\infty}(L, \mathcal{T})$$
be the isomorphism given by the parallel transport along 
the geodesic 
$[0,1] \ni t \mapsto \exp_x (t V_x) \in M$,
where $x \in L$, 
with respect to the Levi-Civita connection of $g$.
Let 
$\perp: TM|_L = TL \oplus \nu \rightarrow \nu$ be the orthogonal projection 
and 
$\nu_V \subset TM |_{\exp_V (L)}$ be the normal bundle of 
$\exp_V (L)$. 
Consider the orthogonal projection 
$$
\perp |_{P_V^{-1}(\nu_V)} : P_V^{-1}(\nu_V) \rightarrow \nu.
$$
The condition for this map to be an isomorphism is open 
and it is an isomorphism for $V=0$. 
Thus shrinking $\mathcal{T}$ if necessary, 
we may assume that 
$$
\phi_V: C^{\infty}(L, \nu_V) \rightarrow C^{\infty}(L, \nu), 
\qquad 
\phi_V(W) = (P_V^{-1}(W))^{\perp}
$$
is an isomorphism for $V \in C^{\infty}(L, \mathcal{T})$. 
Then 
define the first order differential operator 
$F: C^{\infty}(L, \mathcal{T}) \rightarrow C^{\infty}(L, \nu)$ by
\begin{align} \label{def:F}
F(V) = 
\phi_V \left( ({\rm exp}_{V}^{*} \chi) (e_1, e_2, e_3) \right),
\end{align} 
where $\{ e_1, e_2, e_3 \}$ is a local oriented orthonormal frame of $TL$. 
Then 
${\rm exp}_{V}(L) \subset M$ is associative 
if and only if $F(V)=0$. 
Thus 
a neighborhood of $L$ in 
the moduli space of associative submanifolds 
is identified 
with that of $0$ in $F^{-1}(0)$ (in the $C^{1}$ sense). 

Set  
$$D = (dF)_0 : C^{\infty}(L, \nu) \rightarrow C^{\infty}(L, \nu),$$
which 
is the linearization of $F$ at $0$. 
The operator $D$ is computed as follows.

\begin{proposition}[{\cite[Section 5]{Mclean}, \cite[Theorem 2.1]{Gayet}}] \label{prop:explicit D}
Let $(M^7, \varphi, g)$ be a manifold with a $G_2$-structure 
and let $L^3 \subset M^7$ be a compact associative submanifold. 
The operator $D$ above is given by 
$$
D V = \sum_{i=1}^3 e_i \times \nabla^{\perp}_{e_i} V + 
\left( (\nabla_V * \varphi)(e_1, e_2, e_3, \cdot) \right)^{\sharp},
$$
where 
 $\{ e_{1}, e_{2}, e_{3} \}$ is a local oriented orthonormal frame of $TL$ 
satisfying $e_i = e_{i+1} \times e_{i+2}$ for $i \in \mathbb{Z}/3$, 
$\nabla^{\perp}$ is the connection on the normal bundle $\nu$ induced 
by the Levi-Civita connection $\nabla$ of $(M, g)$. 
\end{proposition}

\begin{proof}
For simplicity, we write $\exp_{tV} = \iota_t$. 
Then 
\begin{align*}
D V 
=
\frac{d}{dt} \left( P_{t V}^{-1} (\iota_t^* \chi)(e_1,e_2,e_3)|_{t=0} \right)^{\perp}
=
\left(
\nabla_{\frac{d}{dt}} (\iota_t^* \chi)(e_1,e_2,e_3)|_{t=0}
\right)^{\perp},  
\end{align*}
where $\nabla_{\frac{d}{dt}}$ is the covariant derivative 
along the geodesic 
$[0,1] \ni t \mapsto \exp_x (t V_x) \in M$,
where $x \in L$, 
induced from the Levi-Civita connection of $g$. 
Let $\{ \eta_j \}_{j=1}^7$ be a local orthonormal frame of $TM$.
Then we have 
$$
\chi = - \sum_{j=1}^7 i(\eta_j) * \varphi \otimes \eta_j.
$$
We further compute 
\begin{align*}
D V 
&=
- \sum_j 
\left(
\left.
\nabla_{\frac{d}{dt}} 
\left(
(*\varphi \circ \iota_t) (\eta_j \circ \iota_t, 
(\iota_t)_* e_1, (\iota_t)_* e_2, (\iota_t)_* e_3 
) \eta_j \circ \iota_t 
\right)
\right|_{t=0}
\right)^{\perp}\\
&=
- \sum_j 
\left(
(\nabla_V * \varphi) (\eta_j, e_1, e_2, e_3) \eta_j
+ \sum_{i \in \mathbb{Z}/3} * \varphi (\eta_j, 
\left. \nabla_{\frac{d}{dt}} (\iota_t)_* e_i  \right|_{t=0}, e_{i+1}, e_{i+2}) \eta_j
\right )^{\perp},
\end{align*}
where we use 
$* \varphi (e_1, e_2, e_3, \cdot) = 0$ since $L$ is associative. 
Note that 
$\nabla_{\frac{d}{dt}} (\iota_t)_* e_i$ is the restriction of 
the covariant derivative 
$\nabla_{\frac{d}{dt}} (\bar \iota)_* e_i$ 
along the map $\bar\iota: L \times [0,1] \ni (x,t) \mapsto \iota_t (x) \in M$. 
Then the standard equations of the covariant derivative along the map imply that 
\begin{align*}
\left.
\nabla_{\frac{d}{dt}} (\iota_t)_* e_i \right|_{t=0} 
&=
\left. \nabla_{e_i} (\iota_t)_*  \left( \frac{d}{dt} \right) \right|_{t=0} 
=
\nabla_{e_i} V, \\
* \varphi (\nabla_{e_i} V, e_{i+1}, e_{i+2}, \eta_j)
&=
* \varphi (\nabla^{\perp}_{e_i} V, e_{i+1}, e_{i+2}, \eta_j)\\
&=
g (\chi (\nabla^{\perp}_{e_i} V, e_{i+1}, e_{i+2}), \eta_j)\\
&=
g (- \nabla^{\perp}_{e_i} V \times (e_{i+1} \times e_{i+2}), \eta_j)
=
g (e_i \times \nabla^{\perp}_{e_i} V, \eta_j), 
\end{align*}
where we use the fact that $L$ is associative, (\ref{eq:chicr}) and 
$e_i = e_{i+1} \times e_{i+2}$. 
Then we obtain the statement.
\end{proof}

We can also describe the last term of $D V$ as follows.

\begin{lemma}
By \cite[Section 4]{FG}, we have 
an endomorphism $T \in C^\infty (M, {\rm End}(TM))$ given by 
\begin{align}\label{eq:def of torsion}
\nabla_v \varphi 
=
i(T(v)) * \varphi
\end{align}
for any $v \in TM$. 
Then we have 
$$
\left( (\nabla_V * \varphi)(e_1, e_2, e_3, \cdot) \right)^{\sharp}
= (T(V))^\perp.
$$
\end{lemma}
\begin{proof}
We easily see that 
$
\nabla_v * \varphi 
=
* (\nabla_v \varphi)
=
- (T(v))^\flat \wedge \varphi.
$
Then 
\begin{align*}
\left( (\nabla_V * \varphi)(e_1, e_2, e_3, \cdot) \right)^{\sharp}
&= - \left((T(v))^\flat \wedge \varphi) (e_1, e_2, e_3, \cdot) \right)^{\sharp}\\
&=
\varphi(e_1, e_2, e_3) T(v) 
- \sum_{i \in \mathbb{Z}/3} g(T(v), e_i) \varphi (e_{i+1}, e_{i+2}, \cdot)^\sharp\\
&=
(T(v))^\perp,
\end{align*}
where we use $\varphi(e_1, e_2, e_3) = 1$ and 
$\varphi (e_{i+1}, e_{i+2}, \cdot)^\sharp = e_{i+1} \times e_{i+2} = e_i$. 
\end{proof}

Using this lemma, we see the following. 
\begin{lemma} \label{lem:self-ad}
If $d * \varphi = 0$, $D$ is self-adjoint.  
\end{lemma}

\begin{proof}
For any normal vector fields $V,W \in C^\infty (L, \nu)$, we compute
\begin{align*}
g(D V, W) 
&= g(\sum_{i=1}^3 e_i \times \nabla_{e_i} V + T(V), W)\\
&\stackrel{Lemma \ref{lem:cross asso}}=
\sum_{i=1}^3 g(e_i \times \nabla_{e_i} V, W) + g(T(V), W), \\
\sum_{i=1}^3 g(e_i \times \nabla_{e_i} V, W)
&=
- \sum_{i=1}^3 \varphi (\nabla_{e_i} V, e_i, W)\\
&= 
\sum_{i=1}^3
(- e_i (\varphi (V, e_i, W)) + (\nabla_{e_i} \varphi) (V, e_i, W) \\
&\quad + \varphi (V, \nabla_{e_i} e_i, W) + \varphi (V, e_i, \nabla_{e_i} W) ).
\end{align*}
Define a 1-form $\alpha$ on $L$ by 
$\alpha = \varphi (V, \cdot, W)$. Then 
\begin{align*}
=& d^* \alpha 
+ \sum_{i=1}^3 (\nabla_{e_i} \varphi) (V, e_i, W) + g(V, e_i \times \nabla_{e_i} W)\\
\stackrel{(\ref{eq:def of torsion})}
=&
d^* \alpha 
+ \sum_{i=1}^3 * \varphi (T(e_i), V, e_i, W) + g(V, e_i \times \nabla_{e_i} W).
\end{align*}
By (\ref{eq:chicr}), it follows that 
$$
* \varphi (T(e_i), V, e_i, W) = g (\chi (e_i, V, W), T(e_i))
=
- g(e_i \times (V \times W), T(e_i)).
$$
By Lemma \ref{lem:cross asso}, 
$e_i \times (V \times W)$ is a (local) tangent vector field to $L$. Then 
$$
\sum_{i=1}^3 * \varphi (T(e_i), V, e_i, W) 
=
\sum_{i, j=1}^3 g(T(e_i), e_j) * \varphi (e_j, V, e_i, W).
$$
Hence we obtain
$$
g(D V, W) = 
g(V, D W) + g(T(V), W) - g(V, T(W)) + \sum_{i, j=1}^3 g(T(e_i), e_j) * \varphi (e_j, V, e_i, W) + d^* \alpha.
$$
Then we see that $D$ is self-adjoint if $T$ is symmetric. 
In terms of \cite[Section 2.5]{Kari}, this is 
the case $\nabla \varphi \in W_1 \oplus W_{27}$, 
which is equivalent to $d * \varphi = 0$ by \cite[Table 2.1]{Kari}.
\end{proof}

\begin{remark}
If a $G_2$-structure is torsion-free,
we have 
$T=0$, and hence 
$\left( (\nabla_V * \varphi)(e_1, e_2, e_3, \cdot) \right)^{\sharp} = 0$.
If a $G_2$-structure is nearly parallel $G_2$, 
we have 
$T={\rm id}_{TM}$ and 
$\left( (\nabla_V * \varphi)(e_1, e_2, e_3, \cdot) \right)^{\sharp} = V$. 
We can also deduce this 
by Definition \ref{def:cha_NP} (\cite[Lemma 3.5]{Kdeform}). 
In these cases, $D$ is self-adjoint as stated in Lemma \ref{lem:self-ad}.
\end{remark}

We easily see that 
the operator $D$ is elliptic, and hence Fredholm. 
Since $L$ is $3$-dimensional, the index of $D$ is $0$. 
Thus 
if $D$ is surjective, 
the moduli space of associative submanifolds is $0$-dimensional. 
See
\cite[Proposition 2.2]{Gayet}.

To understand the moduli space of associative submanifolds more, 
we consider their second order deformations in the next subsection.

\subsection{Second order deformations of associative submanifolds}

Use the notation in Section \ref{sec: infinitesimal}. 
The principal task in deformation theory is to integrate given infinitesimal (first order)
deformations $V \in \ker D$. 
Namely, 
to find a one-parameter family $\{ V(t) \} \subset C^{\infty}(L, \nu)$ such that
$$F(V(t))=0 
\qquad 
\mbox{and}  
\qquad
\left. \frac{d}{dt} V(t) \right|_{t=0} = V.$$
In general, this is not possible. 
In this subsection, 
we 
define the second order deformations 
of associative submanifolds 
and 
give a necessary and sufficient condition 
for an infinitesimal associative deformation 
to be integrable (unobstructed) to second order.

\begin{definition}
Let $M^7$ be a manifold with a $G_2$-structure and 
$L^3 \subset M^7$ be a compact associative submanifold. 
An infinitesimal associative deformation $V_1 \in \ker D \subset C^{\infty} (L, \nu)$ 
is said to be 
{\bf unobstructed to second order} 
if there exists $V_2 \in C^{\infty} (L, \nu)$ such that 
$$
\left. \frac{d^2}{d t^2} F \left( t V_1 + \frac{1}{2} t^2 V_2 \right) \right|_{t=0} = 0. 
$$
In other words, 
$t V_1 + \frac{1}{2} t^2 V_2$ gives an associative submanifold 
up to terms of the order $o(t^2)$.
\end{definition}

We easily compute 
$$
\left. \frac{d^2}{d t^2} F \left(t V_1 + \frac{1}{2} t^2 V_2 \right) \right|_{t=0} 
= 
\left. \frac{d^2}{d t^2} F(t V_1) \right|_{t=0} + D (V_2).
$$
Since $D$ is elliptic and $L$ is compact,
we have an orthogonal decomposition 
$C^{\infty} (L, \nu) = {\rm Im} D \oplus {\rm Coker} D$ 
with respect to the $L^2$ inner product. 
Then we obtain the following.

\begin{lemma} \label{lem:unob second order}
Let $\pi: C^{\infty} (L, \nu) \rightarrow {\rm Coker} D$ be an 
orthogonal projection with respect to the $L^2$ inner product. 
Then
an infinitesimal deformation $V_1 \in \ker D$ 
is unobstructed to second order 
if and only if 
\begin{align} \label{eq:unob second order}
\left. \pi \left( \frac{d^2}{d t^2} F(t V_1) \right) \right|_{t=0} = 0.
\end{align}
In other words,
we have 
$
\left \langle \left. \frac{d^2}{d t^2} F(t V_1) \right|_{t=0}, W \right \rangle_{L^2} = 0
$
for any $W \in {\rm Coker} D.$
\end{lemma}

\begin{remark}
Since $D$ is elliptic and hence Fredholm, 
we can construct a Kuranishi model 
for  associative deformations of a compact associative submanifold $L$
(\cite[Section A.4]{MS}) 
as in the case of Lagrangian deformations in 
nearly K\"ahler manifolds (\cite[Theorem 4.10]{LS}). 
Namely, 
there is a real analytic map 
$\tau: \mathfrak{U} \rightarrow \mathfrak{V}$, 
where $\mathfrak{U} \subset \ker D$ 
and 
$\mathfrak{V} \subset {\rm Coker} D$
are open neighborhoods of $0$, 
satisfying $\tau (0) = 0$ and $(d \tau)_0 = 0$ 
such that 
the moduli space of associative deformations of $L$ 
is locally homeomorphic to 
the kernel of $\tau$. 
(Hence the moduli space is locally a finite dimensional analytic variety.)
Then we obtain (\ref{eq:unob second order}) 
by taking the second derivative of $\tau$ at $0$ 
as in \cite[Proposition 4.17]{LS}.
\end{remark}

From Lemma \ref{lem:unob second order}, 
we have to understand $\left. \frac{d^2}{d t^2} F(t V_1) \right|_{t=0}$ 
for the second order deformations. 
It is explicitly computed as follows.

\begin{proposition} \label{prop:second derivative F}
Use the notation in Proposition \ref{prop:explicit D}. 
For $V \in \ker D$, we have
\begin{align*}
\left. \frac{d^2}{dt^2} F(t V) \right |_{t=0}
=& 
((\nabla_V \nabla * \varphi) (V) ) (e_1, e_2, e_3, \cdot)^{\sharp}\\
&+ 
2 \sum_{i \in \mathbb{Z}/3} 
\left((\nabla_V * \varphi) (\nabla^{\perp}_{e_i} V, e_{i+1}, e_{i+2}, \cdot)^{\sharp} \right)^{\perp}\\
&+
\sum_{i=1}^3 e_i \times (R(V, e_i) V)^{\perp}+ 
2 \sum_{i,j=1}^3 g(V, \Pi (e_i, e_j)) e_i \times \nabla^{\perp}_{e_j} V, 
\end{align*}
where $R$ is the curvature tensor of $(M, g)$ and 
$\Pi$ is the second fundamental form of $L$ in $M$. 

If a $G_2$-structure $\varphi$ is torsion-free or nearly parallel $G_2$, 
we have 
\begin{align*}
\left. \frac{d^2}{dt^2} F(t V) \right |_{t=0}
= 
\sum_{i=1}^3 e_i \times (R(V, e_i) V)^{\perp} 
+
2 \sum_{i,j=1}^3 g(V, \Pi (e_i, e_j)) e_i \times \nabla^{\perp}_{e_j} V.
\end{align*}
\end{proposition}

\begin{proof}
Use the notation in the proof of Proposition \ref{prop:explicit D}. 
Setting 
$\chi_j = - i(\eta_j) * \varphi$, we have 
$\chi = \sum_{j=1}^7 \chi_j \otimes \eta_j$. 
Then 
\begin{align}
\frac{d}{dt} F(t V)
=&
\left(
P_{t V}^{-1}
\nabla_{\frac{d}{dt}} \left( (\iota_t^* \chi)(e_1,e_2,e_3) \right)
\right)^\perp, \nonumber \\
\left. \frac{d^2}{dt^2} F(t V) \right |_{t=0}
=&
\left.
\left(
\nabla_{\frac{d}{dt}}
\nabla_{\frac{d}{dt}} \left( (\iota_t^* \chi)(e_1,e_2,e_3) \right)
\right)^{\perp} \right |_{t=0} \nonumber \\
=&
\left.
\sum_j 
\left(\nabla_{\frac{d}{dt}}
\nabla_{\frac{d}{dt}} 
(\iota_t^* \chi_j (e_1, e_2, e_3) \eta_j \circ \iota_t) \right)^{\perp} \right|_{t=0} \nonumber \\
=&
\sum_j 
\left(
\left. \frac{d^2}{dt^2} \iota_t^* \chi_j (e_1, e_2, e_3) \right|_{t=0} \eta_j
+
2 \left. \frac{d}{dt} \iota_t^* \chi_j (e_1, e_2, e_3) \right|_{t=0} \nabla_V \eta_j \right. 
\nonumber \\
&+
\left. \left. \chi_j (e_1, e_2, e_3)  \nabla_{\frac{d}{dt}} \nabla_{\frac{d}{dt}} \eta_j \circ \iota_t \right|_{t=0}
\right)^{\perp}. \label{eq:second derivative F 1}
\end{align}
Since 
$\left. \frac{d}{dt} \iota_t^* \chi_j (e_1, e_2, e_3) \right|_{t=0}
= g(D V, \eta_j)$
by the proof of Proposition \ref{prop:explicit D}
and $\chi_j (e_1, e_2, e_3) = 0$, 
we only have to compute 
$\left. \frac{d^2}{dt^2} \iota_t^* \chi_j (e_1, e_2, e_3) \right|_{t=0}$.
Then 
\begin{align*}
\left. \frac{d^2}{dt^2} \iota_t^* \chi_j (e_1, e_2, e_3) \right|_{t=0}
=&
- \left. \frac{d^2}{dt^2} \left( (* \varphi \circ \iota_t) (\eta_j \circ \iota_t, 
(\iota_t)_* e_1, (\iota_t)_* e_2, (\iota_t)_* e_3) \right) \right|_{t=0} \\
=&
- 
\frac{d}{dt}
\left( (\nabla_{\frac{d}{dt}} * \varphi \circ \iota_t) (\eta_j \circ \iota_t, 
(\iota_t)_* e_1, (\iota_t)_* e_2, (\iota_t)_* e_3) \right.\\
&+
(* \varphi \circ \iota_t) (\nabla_{\frac{d}{dt}} \eta_j \circ \iota_t, 
(\iota_t)_* e_1, (\iota_t)_* e_2, (\iota_t)_* e_3)\\
&\left. \left. +
\sum_{i \in \mathbb{Z}/3}
(* \varphi \circ \iota_t) (\eta_j \circ \iota_t, 
\nabla_{\frac{d}{dt}} (\iota_t)_* e_i, (\iota_t)_* e_{i+1}, (\iota_t)_* e_{i+2}) \right) \right|_{t=0}\\
=&
- (\nabla_{\frac{d}{dt}} \nabla_{\frac{d}{dt}} * \varphi \circ \iota_t) |_{t=0}
 (\eta_j, e_1, e_2, e_3) \\
&- 2 \sum_{i \in \mathbb{Z}/3}  (\nabla_V * \varphi) (\eta_j, 
\nabla_{\frac{d}{dt}} (\iota_{t})_* e_i |_{t=0}, e_{i+1}, e_{i+2})\\
&- 
* \varphi (\nabla_{\frac{d}{dt}} \nabla_{\frac{d}{dt}} \eta_j \circ \iota_t |_{t=0}, 
e_1,e_2,e_3)\\
&-2 (\nabla_V * \varphi) (\nabla_V \eta_j, e_1, e_2, e_3)\\
&-2
\sum_{i \in \mathbb{Z}/3}
* \varphi (\nabla_V \eta_j, 
\nabla_{\frac{d}{dt}} (\iota_{t})_* e_i |_{t=0}, e_{i+1}, e_{i+2})\\
&- 
\sum_{i \in \mathbb{Z}/3}
* \varphi (\eta_j, 
\nabla_{\frac{d}{dt}} \nabla_{\frac{d}{dt}} (\iota_{t})_* e_i |_{t=0}, e_{i+1}, e_{i+2})\\
&-2 
\sum_{i \in \mathbb{Z}/3}
* \varphi (\eta_j, 
\nabla_{\frac{d}{dt}} (\iota_{t})_* e_i |_{t=0}, \nabla_{\frac{d}{dt}} (\iota_{t})_* e_{i+1} |_{t=0} , e_{i+2}).
\end{align*}
By the same argument as in the proof of Proposition \ref{prop:explicit D}, we have 
\begin{align*}
* \varphi (e_1, e_2, e_3, \cdot) &= 0, \\
\nabla_{\frac{d}{dt}} (\iota_{t})_* e_i |_{t=0} &= \nabla_{e_i} V, \\
\nabla_{\frac{d}{dt}} \nabla_{\frac{d}{dt}} (\iota_{t})_* e_i |_{t=0}
&=
R(V, e_i) V + \left. \nabla_{e_i} \nabla_{\frac{d}{dt}} (\iota_{t})_{*} \left(\frac{d}{dt} \right) \right|_{t=0}
=
R(V, e_i) V,
\end{align*}
where we use $\nabla_{\frac{d}{dt}} (\iota_{t})_{*} \left(\frac{d}{dt} \right) = 0$
because $\iota_t = \exp (t V)$ is a geodesic.
By the definition of the induced connection, we have

\begin{align*}
\left. \nabla_{\frac{d}{dt}} \nabla_{\frac{d}{dt}} * \varphi \circ \iota_t \right|_{t=0}
&=
\left. \nabla_{\frac{d}{dt}} ((\nabla_{\frac{d \iota_t}{dt}} * \varphi) \circ \iota_t)  \right|_{t=0}\\
&=
\left. 
\nabla_{\frac{d}{dt}}\left( ((\nabla * \varphi) \circ \iota_t) \left(\frac{d \iota_t}{dt} \right) \right) 
\right|_{t=0}
=
(\nabla_V \nabla * \varphi)(V), 
\end{align*}
where we use $\nabla_{\frac{d}{dt}} \frac{d \iota_{t}}{dt} = 0$.
Moreover, 
by the proof of Proposition \ref{prop:explicit D}
we have 
\begin{align*}
(\nabla_V * \varphi) (\nabla_V \eta_j, e_1, e_2, e_3) 
+
\sum_{i \in \mathbb{Z}/3}
* \varphi ( \nabla_V \eta_j, \nabla_{\frac{d}{dt}} (\iota_t)_* e_i |_{t=0}, e_{i+1}, e_{i+2})
= 
- g(D V, \nabla_V \eta_j).
\end{align*}
Thus it follows that 
\begin{align}
\left. \frac{d^2}{dt^2} \iota_t^* \chi_j (e_1, e_2, e_3) \right|_{t=0}
= &
- ((\nabla_V \nabla * \varphi) (V)) (\eta_j, e_1, e_2, e_3) \nonumber \\
&-2 \sum_{i \in \mathbb{Z}/3} (\nabla_V * \varphi) (\eta_j, \nabla_{e_i} V, e_{i+1}, e_{i+2}) \nonumber \\
&- \sum_{i \in \mathbb{Z}/3} * \varphi (\eta_j, R(V,e_i) V, e_{i+1}, e_{i+2}) \nonumber \\
&-2 \sum_{i \in \mathbb{Z}/3} * \varphi (\eta_j, \nabla_{e_i} V, \nabla_{e_{i+1}} V, e_{i+2}) 
+ 2 g(D V, \nabla_V \eta_j). \label{eq:second derivative F 2}
\end{align}
Hence 
from (\ref{eq:chicr}), (\ref{eq:second derivative F 1}) and (\ref{eq:second derivative F 2}), 
we obtain
\begin{align}
\left. \frac{d^2}{dt^2} F(t V) \right |_{t=0}
=&
((\nabla_V \nabla * \varphi) (V) ) (e_1, e_2, e_3, \cdot) ^{\sharp} \nonumber\\
&+ 
2 \sum_{i \in \mathbb{Z}/3} 
\left((\nabla_V * \varphi) (\nabla_{e_i} V, e_{i+1}, e_{i+2}, \cdot)^{\sharp} \right)^{\perp} \nonumber\\
&+
\sum_{i=1}^3 e_i \times (R(V, e_i) V)^{\perp}  \nonumber \\
&+
2 \sum_{i \in \mathbb{Z}/3} \chi (\nabla_{e_i} V, \nabla_{e_{i+1}} V, e_{i+2})^{\perp} \nonumber \\
&
+ 2 \sum_j g(D V, \nabla_V \eta_j) \eta_j^{\perp}
+ 2 \sum_j g(D V, \eta_j) (\nabla_V \eta_j)^{\perp}.  \label{eq:second derivative F 3}
\end{align}

Next, we compute 
$\sum_{i \in \mathbb{Z}/3} \chi (\nabla_{e_i} V, \nabla_{e_{i+1}} V, e_{i+2})^{\perp}$.
Let $\top: TM|_L \rightarrow TL$ be the projection.
Since $L$ is associative, we have
\begin{align*}
&\chi (\nabla_{e_i} V, \nabla_{e_{i+1}} V, e_{i+2})^{\perp}\\
=&
\chi (\nabla^{\top}_{e_i} V, \nabla^{\perp}_{e_{i+1}} V, e_{i+2})^{\perp}
+
\chi (\nabla^{\perp}_{e_i} V, \nabla^{\top}_{e_{i+1}} V, e_{i+2})^{\perp}
+
\chi (\nabla^{\perp}_{e_i} V, \nabla^{\perp}_{e_{i+1}} V, e_{i+2})^{\perp}. 
\end{align*}

The first term is computed as 
\begin{align*}
\chi (\nabla^{\top}_{e_i} V, \nabla^{\perp}_{e_{i+1}} V, e_{i+2})
=&
- \sum_{j=0}^2 g(V, \Pi (e_i, e_{i+j})) \chi (e_{i+j}, \nabla^{\perp}_{e_{i+1}} V, e_{i+2}) \\
\stackrel{(\ref{eq:chicr})}
=&
- \sum_{j=0}^2 g(V, \Pi (e_i, e_{i+j})) \nabla^{\perp}_{e_{i+1}} V \times (e_{i+j} \times e_{i+2}) \\
=&
- g(V, \Pi (e_i, e_i)) e_{i+1} \times \nabla^{\perp}_{e_{i+1}} V
+ g(V, \Pi (e_i, e_{i+1})) e_i \times \nabla^{\perp}_{e_{i+1}} V.
\end{align*}
The second term is computed as 
\begin{align*}
\chi (\nabla^{\perp}_{e_i} V, \nabla^{\top}_{e_{i+1}} V, e_{i+2})
=&
- \sum_{j=0}^2 g(V, \Pi (e_{i+1}, e_{i+j})) \chi (\nabla^{\perp}_{e_i} V, e_{i+j}, e_{i+2}) \\
\stackrel{(\ref{eq:chicr})}
=&
\sum_{j=0}^2 g(V, \Pi (e_{i+1}, e_{i+j})) \nabla^{\perp}_{e_i} V \times (e_{i+j} \times e_{i+2}) \\
=&
g(V, \Pi (e_i, e_{i+1})) e_{i+1} \times \nabla^{\perp}_{e_i} V
- g(V, \Pi (e_{i+1}, e_{i+1})) e_i \times \nabla^{\perp}_{e_i} V.
\end{align*}
The third term is computed as 
$$
\chi (\nabla^{\perp}_{e_i} V, \nabla^{\perp}_{e_{i+1}} V, e_{i+2})
= \chi (e_{i+2}, \nabla^{\perp}_{e_i} V, \nabla^{\perp}_{e_{i+1}} V)
\stackrel{(\ref{eq:chicr})}
=
- e_{i+2} \times (\nabla^{\perp}_{e_i} V \times \nabla^{\perp}_{e_{i+1}} V), 
$$
which is a section of $TL$ by Lemma \ref{lem:cross asso}. Then 
\begin{align*}
\chi (\nabla^{\perp}_{e_i} V, \nabla^{\perp}_{e_{i+1}} V, e_{i+2})^{\perp}
= 0.
\end{align*}
Hence we obtain
\begin{align}
&\sum_{i \in \mathbb{Z}/3} \chi (\nabla_{e_i} V, \nabla_{e_{i+1}} V, e_{i+2})^{\perp}
\nonumber\\
=&
- 
\sum_{i \in \mathbb{Z}/3} 
\left \{ g(V, \Pi (e_{i+1}, e_{i+1})) + g(V, \Pi (e_{i+2}, e_{i+2})) \right \} e_i \times \nabla^{\perp}_{e_i} V 
\nonumber\\
&+
\sum_{i \in \mathbb{Z}/3} 
g(V, \Pi (e_i, e_{i+1})) \left( e_i \times \nabla^{\perp}_{e_{i+1}} V + e_{i+1} \times \nabla^{\perp}_{e_i} V \right) 
\nonumber \\
=& 
\sum_{i,j=1}^3
g(V, \Pi (e_i, e_j)) e_i \times \nabla^{\perp}_{e_j} V 
- 
\left(\sum_{i=1}^3 g(V, \Pi (e_i, e_i)) \right) 
\left(\sum_{j=1}^3 e_j \times \nabla_{e_j}^{\perp} V \right).
\label{eq:second derivative F 4}
\end{align}
Thus 
using the equation 
\begin{align*}
&\sum_{i \in \mathbb{Z}/3} (\nabla_V * \varphi) (\nabla_{e_i} V, e_{i+1}, e_{i+2}, \cdot)\\
=&
\sum_{i \in \mathbb{Z}/3} (\nabla_V * \varphi) (\nabla^{\perp}_{e_i} V, e_{i+1}, e_{i+2}, \cdot)
- 
\left(\sum_{i=1}^3 g(V, \Pi (e_i, e_i)) \right) 
(\nabla_V * \varphi) (e_1, e_2, e_3, \cdot), 
\end{align*}
we obtain 
from Proposition \ref{prop:explicit D}, 
(\ref{eq:second derivative F 3}) and (\ref{eq:second derivative F 4}) 
\begin{align*}
\left. \frac{d^2}{dt^2} F(t V) \right |_{t=0}
=&
((\nabla_V \nabla * \varphi) (V) ) (e_1, e_2, e_3, \cdot) ^{\sharp} \\
&+ 
2 \sum_{i \in \mathbb{Z}/3} 
\left((\nabla_V * \varphi) (\nabla^{\perp}_{e_i} V, e_{i+1}, e_{i+2}, \cdot)^{\sharp} \right)^{\perp} \\
&+
\sum_{i=1}^3 e_i \times (R(V, e_i) V)^{\perp} 
+ 2 \sum_{i,j=1}^3  
g(V, \Pi (e_i, e_j)) e_i \times \nabla^{\perp}_{e_j} V 
\\
&+ 2 \sum_j g(D V, \nabla_V \eta_j) \eta_j^{\perp}
+ 2 \sum_j g(D V, \eta_j) (\nabla_V \eta_j)^{\perp} \\
&-2 \left(\sum_{i=1}^3 g(V, \Pi (e_i, e_i)) \right) D V, 
\end{align*}
which implies the first equation of Proposition \ref{prop:second derivative F}.

If a $G_2$-structure $\varphi$ is torsion-free, 
the second equation of 
Proposition \ref{prop:second derivative F} is obvious. 
If $\varphi$ is nearly parallel $G_2$, we have by Definition \ref{def:cha_NP} 
\begin{align*}
(\nabla_V \nabla * \varphi) (V) 
&= 
\nabla_V \nabla_V * \varphi - \nabla_{\nabla_V V} * \varphi \\
&=
\nabla_V (-V^\flat \wedge \varphi) + (\nabla_V V)^\flat \wedge \varphi \\
&=
-V^\flat \wedge i(V) * \varphi,
\end{align*}
which implies that 
$$
((\nabla_V \nabla * \varphi) (V))(e_1,e_2,e_3, \cdot) = 0. 
$$
Similarly, we have by Definition \ref{def:cha_NP} 
\begin{align*}
(\nabla_V * \varphi) (\nabla^{\perp}_{e_i} V, e_{i+1}, e_{i+2}, \cdot)^{\sharp}
=
- (V^\flat \wedge \varphi) (\nabla^{\perp}_{e_i} V, e_{i+1}, e_{i+2}, \cdot) ^{\sharp} 
=
- g(V, \nabla^{\perp}_{e_i} V) e_i, 
\end{align*}
where we use $\varphi (e_{i+1}, e_{i+2}, \cdot) ^{\sharp} = e_{i+1} \times e_{i+2} = e_i.$
Hence  
$$
\sum_{i \in \mathbb{Z}/3} 
\left((\nabla_V * \varphi) (\nabla_{e_i} V, e_{i+1}, e_{i+2}, \cdot)^{\sharp} \right)^{\perp} = 0.
$$
Thus we obtain  the second equation of 
Proposition \ref{prop:second derivative F}. 
\end{proof}

\begin{remark}
Using the endomorphism $T$ given by (\ref{eq:def of torsion}), 
we have 
\begin{align*}
&((\nabla_V \nabla * \varphi) (V) ) (e_1, e_2, e_3, \cdot)^{\sharp} 
+ 
2 \sum_{i \in \mathbb{Z}/3} 
\left((\nabla_V * \varphi) (\nabla^{\perp}_{e_i} V, e_{i+1}, e_{i+2}, \cdot)^{\sharp} \right)^{\perp}\\
=&
\left( (\nabla_V T)(V) \right)^\perp
+ T(V)^\top \times T(V)^\perp -2 \nabla^\perp_{T(V)^\top} V
\end{align*}
by a direct computation. 
If $\varphi$ is torsion-free ($T = 0$) these terms obviously vanish. 
If $\varphi$ is nearly parallel $G_2$ ($T= {\rm id}_{TM}$), 
these terms vanish again 
because 
$\nabla {\rm id}_{TM}=0$ and $({\rm id}_{TM} (V))^\top = 0$ for a normal vector field $V$. 
\end{remark}

\section{Associative submanifolds in $S^7$} \label{sec:computation S7}

In this section, we give a proof of Theorem \ref{thm:A3 unob to second order}.
The standard 7-sphere $S^7$ has a natural nearly parallel $G_2$-structure 
(\cite[Section 2]{Lotay3}). 
Homogeneous associative submanifolds in $S^7$ 
are  classified by Lotay (\cite[Theorem 1.1]{Lotay3}). 
As noted in the introduction, 
there is a mysterious  homogeneous example 
called $A_3$ 
which does not arise from other geometries.

First, 
we summarize the facts for $A_3$ from \cite[Example 6.3, Section 6.3.3]{Kdeform}. 
Define $\rho_3: {\rm SU}(2) \hookrightarrow {\rm SU}(4)$ by 
\begin{align} \label{eq: irr SU(2) on C4}
\rho_3 
\left(
\left( 
\begin{array}{cc}
a & -\overline{b} \\
b &   \overline{a} \\
\end{array} 
\right)
\right)
= 
\left( 
\begin{array}{cccc}
a^{3}               & -\sqrt{3} a^{2} \overline{b} & \sqrt{3} a \overline{b}^{2}        & - \overline{b}^{3}\\
\sqrt{3} a^{2}b  & a (|a|^{2} - 2 |b|^{2})         & - \overline{b} (2|a|^{2} - |b|^{2}) & \sqrt{3} \overline{a} \overline{b}^{2}\\
\sqrt{3} a b^{2} & b (2|a|^{2} - |b|^{2})          & \overline{a} (|a|^{2} - 2 |b|^{2})  &-\sqrt{3} \overline{a}^{2} \overline{b}\\
b^{3}               & \sqrt{3} \overline{a} b^{2}  & \sqrt{3} \overline{a}^{2} b        & \overline{a}^{3}
\end{array} 
\right),  
\end{align}
where  
$a, b \in \mathbb{C}$ such that $|a|^{2} + |b|^{2} = 1$. 
This is an irreducible ${\rm SU}(2)$-action on $\mathbb{C}^4$. 
By using the notation of Appendix \ref{sec:R irr rep}, 
$\rho_3$ is the matrix representation 
of $\rho_3: {\rm SU}(2) \rightarrow {\rm GL}(V_3) \cong {\rm GL}(4, \mathbb{C})$ 
with respect to 
the basis 
$\{ v^{(3)}_0, \cdots, v^{(3)}_3 \}$.
Then 
\begin{align} \label{eq:def A3}
A_{3} = \rho_3 ({\rm SU}(2)) \cdot \frac{1}{\sqrt{2}} {}^t\! (0, 1, i, 0) \cong {\rm SU}(2)
\end{align}
is an associative submanifold in $S^7$.

Define the basis of the Lie algebra $\mathfrak{su}(2)$ of ${\rm SU}(2)$ by 
\begin{align} \label{E1E2E3}
E_{1} = 
\left( 
\begin{array}{cc}
0   & 1 \\
-1 & 0
\end{array} 
\right), \qquad
E_{2} = 
\left( 
\begin{array}{cc}
0 & i \\
i & 0
\end{array} 
\right), \qquad
E_{3} = 
\left( 
\begin{array}{cc}
i &   0 \\
0 & -i
\end{array} 
\right), 
\end{align}
which satisfies the relation $[E_{i}, E_{i + 1}] = 2 E_{i + 2}$ for  $i \in \mathbb{Z}/3$. 
Denote by $e_{1}, e_{2}, e_{3}$
the left invariant vector fields 
on ${\rm SU}(2) \cong A_{3}$ 
induced by 
$\frac{1}{\sqrt{7}} E_{1}, \frac{1}{\sqrt{7}} E_{2}, E_{3}$, respectively. 
Then they define a global orthonormal frame of $T A_3$. 
Explicitly, we have 
at $p_{0} = \frac{1}{\sqrt{2}} {}^t\! (0, 1, i, 0)$
\begin{align*}
e_1 = 
\frac{1}{\sqrt{14}} {}^t\! \left(\sqrt{3}, 2i, -2, -\sqrt{3}i \right), \ 
e_2 =
\frac{1}{\sqrt{14}} 
{}^t\! \left(\sqrt{3}i, -2, 2i, -\sqrt{3} \right), \ 
e_3 =
\frac{1}{\sqrt{2}} 
{}^t\! \left(0, i, 1, 0 \right), 
\end{align*}
and $(e_i)_{\rho_3 (g) \cdot p_0} = \rho_3 (g) \cdot (e_i)_{p_0}$ 
for $g \in {\rm SU}(2)$.
Set 
\begin{align} \label{eq:def of eta i}
\begin{aligned}
(\eta_{1})_{p_0} 
=& \frac{1}{\sqrt{2}} {}^t\! \left(i, 0, 0, 1 \right), \qquad
&
(\eta_{3})_{p_0} 
=& \frac{1}{\sqrt{42}} {}^t\! \left(-2 \sqrt{3}i, -3, 3i, 2 \sqrt{3} \right), \\
(\eta_{2})_{p_0} 
=& \frac{1}{\sqrt{2}} {}^t\! \left(-1, 0, 0, -i \right), \qquad
&
(\eta_{4})_{p_0} 
=& \frac{1}{\sqrt{42}} {}^t\! \left(-2 \sqrt{3}, 3 i, -3, 2 \sqrt{3} i \right), 
\end{aligned}
\end{align}
which is an orthonormal basis of the normal bundle at $p_0$. 
Setting  
 $(\eta_j)_{\rho_3 (g) \cdot p_0} = \rho_3 (g) \cdot (\eta_j)_{p_0}$ 
for $g \in {\rm SU}(2)$, 
we obtain an orthonormal frame $\{ \eta_j \}_{j=1}^4$ of the normal bundle $\nu$.

\subsection{Second order deformations of $A_3$}

Now, 
we consider the second order deformations of $A_3$. 
First, we describe 
the second derivative of the deformation map 
in a normal direction $V \in C^{\infty}(A_3, \nu)$
explicitly 
using Proposition \ref{prop:second derivative F}. 
Since $S^7$
with the round metric $\langle \cdot, \cdot \rangle$ 
has constant sectional curvature 1, 
we have 
$$
R(x,y)z = \langle y,z \rangle x - \langle x,z \rangle y \qquad \mbox{ for } x,y,z \in TS^7, 
$$
which implies that 
$$
(R(V, e_i) V)^{\perp} = 0.
$$
Then by Proposition \ref{prop:second derivative F}, it follows that 
\begin{align*}
\left. \frac{d^2}{dt^2} F(t V) \right |_{t=0}
= 
2 \sum_{i, j=1}^3 \langle V, \Pi (e_i, e_j) \rangle e_i \times \nabla^{\perp}_{e_j} V.
\end{align*}
We will compute this. 
By \cite[Lemma 6.20]{Kdeform} and its proof, 
we have the following. 
\begin{lemma} \label{lem:eqn for A3}
\begin{align*}
(\nabla^{\perp}_{e_{i}} \eta_{j})_{1 \leq i \leq 3, 1 \leq j \leq 4} 
&=
\frac{3}{7}
\left( 
\begin{array}{cccc}
- \eta_{4} & -\eta_{3}   & \eta_{2}    & \eta_{1} \\
 \eta_{3}  & - \eta_{4}  & - \eta_{1}  & \eta_{2} \\
7\eta_{2}  & -7 \eta_{1} & -5 \eta_{4} & 5 \eta_{3} 
\end{array} 
\right), \\
(e_{i} \times \eta_{j})_{1 \leq i \leq 3, 1 \leq j \leq 4} 
&= 
\left( 
\begin{array}{cccc}
\eta_{4} & \eta_{3}  & -\eta_{2} & -\eta_{1} \\
-\eta_{3} & \eta_{4}   & \eta_{1}  & -\eta_{2} \\
\eta_{2}  & -\eta_{1} & \eta_{4}   & -\eta_{3} 
\end{array} 
\right), \\
\left( \Pi (e_i, e_j) \right)_{1 \leq i,j \leq 3}
&=
\frac{2 \sqrt{3}}{7} 
\left(
\begin{array}{ccc}
\eta_1     & \eta_2    & -2 \eta_3 \\
\eta_2     & - \eta_1 &   2 \eta_4 \\
-2 \eta_3 & 2 \eta_4 & 0  \\
\end{array} 
\right).
\end{align*}
\end{lemma}
Then $\left. \frac{d^2}{dt^2} F(t V) \right |_{t=0}$ 
is described explicitly as follows.
\begin{lemma} \label{lem:explicit second der A3}
Set
$$
V = \sum_{j=1}^4 V_j \eta_j \in \ker D, \qquad 
\left. \frac{d^2}{dt^2} F(t V) \right |_{t=0} 
= 
 \sum_{j=1}^4 F_j \eta_j, 
$$
where 
$V_j, F_j \in C^{\infty}(A_3)$ are smooth functions on $A_3$. 
Denoting 
$\mathcal{V}_1 = V_1 + i V_2$, 
$\mathcal{V}_2 = V_3 - i V_4$, 
we have 
\begin{align*}
F_1 + i F_2
=&
\frac{4 \sqrt{3}}{7} \left \{ -(i e_1 + e_2) (\mathcal{V}_1 \mathcal{V}_2) 
+ \bar{\mathcal{V}}_2 (- i e_1 + e_2) \mathcal{V}_1 
+ \left(i e_3 - \frac{24}{7} \right) (\mathcal{V}_2^2) \right \}, \\
F_3 - i F_4
=&
\frac{4 \sqrt{3}}{7} \left \{ \bar{\mathcal{V}}_1 (-i e_1 + e_2) \mathcal{V}_1
+ \frac{1}{2} (i e_1 + e_2) (\mathcal{V}_2^2) \right. \\
&
\left. \qquad 
+\bar{\mathcal{V}}_2 \left( \left( 2 i e_3 - \frac{48}{7} \right)\mathcal{V}_1 
-(-i e_1 + e_2) \mathcal{V}_2  \right)  \right \}.
\end{align*}
\end{lemma}

\begin{proof}
By the third equation of Lemma \ref{lem:eqn for A3}, we have
\begin{align*}
&\sum_{i,j=1}^3 \langle V, \Pi (e_i, e_j) \rangle e_i \times \nabla^{\perp}_{e_j} V\\
=&
\frac{2 \sqrt{3}}{7} V_1 (e_1 \times \nabla^{\perp}_{e_1} V - e_2 \times \nabla^{\perp}_{e_2} V) 
+
\frac{2 \sqrt{3}}{7} V_2 (e_1 \times \nabla^{\perp}_{e_2} V + e_2 \times \nabla^{\perp}_{e_1} V)\\
&-
\frac{4 \sqrt{3}}{7} V_3 (e_1 \times \nabla^{\perp}_{e_3} V + e_3 \times \nabla^{\perp}_{e_1} V)
+
\frac{4 \sqrt{3}}{7} V_4 (e_2 \times \nabla^{\perp}_{e_3} V + e_3 \times \nabla^{\perp}_{e_2} V).
\end{align*}
By the first and the second equations of Lemma \ref{lem:eqn for A3}, we have 
\begin{align*}
\nabla^{\perp}_{e_1} V 
&= \sum_{j=4}^7 e_1 (V_j) \eta_j 
+ \frac{3}{7} \left( -V_1 \eta_4 - V_2 \eta_3 + V_3 \eta_2 + V_4 \eta_1 \right), \\
\nabla^{\perp}_{e_2} V 
&= \sum_{j=4}^7 e_2 (V_j) \eta_j 
+ \frac{3}{7} \left( V_1 \eta_3 - V_2 \eta_4 - V_3 \eta_1 + V_4 \eta_2 \right), \\
\nabla^{\perp}_{e_3} V 
&= \sum_{j=4}^7 e_3 (V_j) \eta_j 
+ \frac{3}{7} \left( 7 V_1 \eta_2 - 7 V_2 \eta_1 - 5 V_3 \eta_4 + 5 V_4 \eta_3 \right).
\end{align*}
Then by the second equation of Lemma \ref{lem:eqn for A3} 
and a straightforward computation, we obtain 
\begin{align*}
&\left. \frac{d^2}{dt^2} F(t V) \right |_{t=0} \\
=&
\frac{4 \sqrt{3}}{7} V_1 
\left \{ (-e_1(V_4) - e_2 (V_3)) \eta_1 + (-e_1(V_3) + e_2 (V_4)) \eta_2 \right. \\
   &\qquad \qquad \left. + (e_1(V_2) + e_2 (V_1)) \eta_3 + (e_1(V_1) -e_2 (V_2)) \eta_4
\right \}\\
&+
\frac{4 \sqrt{3}}{7} V_2
\left \{ (-e_2(V_4) + e_1 (V_3)) \eta_1 + (-e_2(V_3) - e_1 (V_4)) \eta_2 \right. \\
   &\qquad \qquad \left. + (e_2(V_2) - e_1 (V_1)) \eta_3 + (e_2(V_1) + e_1 (V_2)) \eta_4
\right \}\\
&-
\frac{8 \sqrt{3}}{7} V_3
\left \{ \left(-e_3(V_4) - e_1 (V_2) + \frac{12}{7} V_3 \right) \eta_1 
          + \left(-e_3(V_3) + e_1 (V_1) - \frac{12}{7} V_4 \right) \eta_2 \right. \\
   &\qquad \qquad \left. + \left(e_3(V_2) - e_1 (V_4) + \frac{24}{7} V_1 \right) \eta_3 
           + \left(e_3(V_1) + e_1 (V_3) - \frac{24}{7} V_2 \right) \eta_4
\right \}\\
&+
\frac{8 \sqrt{3}}{7} V_4 
\left \{ \left(e_3(V_3) - e_2 (V_2) + \frac{12}{7} V_4 \right) \eta_1 
          + \left(-e_3(V_4) + e_2 (V_1) + \frac{12}{7} V_3 \right) \eta_2 \right. \\
   &\qquad \qquad \left. + \left(-e_3(V_1) - e_2 (V_4) + \frac{24}{7} V_2 \right) \eta_3 
           + \left(e_3(V_2) + e_2 (V_3) + \frac{24}{7} V_1 \right) \eta_4
\right \}.
\end{align*}
Hence 
\begin{align*}
F_1+i F_2
=& \frac{4 \sqrt{3}}{7} 
V_1 \left( -e_1 (V_4 + i V_3) -e_2 (V_3 -i V_4) \right)\\
&+
\frac{4 \sqrt{3}}{7} 
V_2 \left( -e_2 (V_4 + i V_3) +e_1 (V_3 -i V_4) \right)\\
&- 
\frac{8 \sqrt{3}}{7} 
V_3 \left( -e_3 (V_4 + i V_3) +e_1 (i V_1 - V_2) + \frac{12}{7} (V_3 -i V_4) \right)\\
&+
\frac{8 \sqrt{3}}{7} 
V_4 \left( e_3 (V_3 - i V_4) +e_2 (i V_1 - V_2) + \frac{12}{7} (V_4 +i V_3)  \right) \\
=&
\frac{4 \sqrt{3}}{7} 
V_1  (-i e_1 -e_2) (V_3 -i V_4) 
+
\frac{4 \sqrt{3}}{7} 
V_2  (e_1 -i e_2) (V_3 -i V_4)  \\
&- 
\frac{8 \sqrt{3}}{7} 
V_3   \left( i e_1 (V_1 + i V_2) + \left( -i e_3 + \frac{12}{7} \right) (V_3 -i V_4) \right)\\
&+
\frac{8 \sqrt{3}}{7} 
V_4   \left( i e_2 (V_1 + i V_2) + \left( e_3 + \frac{12}{7} i \right) (V_3 -i V_4) \right), 
\end{align*}
\begin{align*}
F_3-i F_4
=& \frac{4 \sqrt{3}}{7} 
V_1 \left( e_1 (V_2 - i V_1) +e_2 (V_1 +i V_2) \right)\\
&+
\frac{4 \sqrt{3}}{7} 
V_2 \left( e_2 (V_2 - i V_1) -e_1 (V_1 +i V_2) \right)\\
&- 
\frac{8 \sqrt{3}}{7} 
V_3 \left( e_3 (V_2 - i V_1) - e_1 (V_4 +i V_3) + \frac{24}{7} (V_1 +i V_2) \right)\\
&+
\frac{8 \sqrt{3}}{7} 
V_4 \left( -e_3 (V_1 + i V_2) -e_2 (V_4 + i V_3) + \frac{24}{7} (V_2 -i V_1)  \right)\\
=&
\frac{4 \sqrt{3}}{7} 
V_1  (-i e_1 +e_2) (V_1 +i V_2) 
+
\frac{4 \sqrt{3}}{7} 
V_2  (-e_1 -i e_2) (V_1 +i V_2)  \\
&- 
\frac{8 \sqrt{3}}{7} 
V_3   \left( \left( -i e_3 + \frac{24}{7} \right) (V_1 +i V_2) -i e_1 (V_3 - i V_4) \right)\\
&+
\frac{8 \sqrt{3}}{7} 
V_4   \left( \left( -e_3 - \frac{24}{7} i \right) (V_1+i V_2) - i e_2 (V_3 - i V_4) \right).
\end{align*}
Using 
$2i (-V_3 e_1 + V_4 e_2) = - \mathcal{V}_2 (i e_1 + e_2) + 
\bar{\mathcal{V}}_2 (- i e_1 + e_2)$, 
we obtain the statement. 
\end{proof}

By \cite[(6.24),(6.25)]{Kdeform} and the proof of \cite[Proposition 6.22]{Kdeform}, 
we know the following about $\ker D$, 
where $D$ is given in Proposition \ref{prop:explicit D}. 
Note that $D$ in this paper corresponds to $D + {\rm id}_{\nu}$ in \cite{Kdeform}. 

\begin{lemma} \label{lem:DV=0}
For $V = \sum_{j=1}^4 V_j \eta_j \in C^\infty (L, \nu)$, 
set $\mathcal{V}_1 = V_1 + i V_2$ and $\mathcal{V}_2 = V_3 - i V_4$. 
Then $D V=0$ is equivalent to 
\begin{align} \label{eq:DV=0}
\begin{split}
\left(i e_3 - \frac{8}{7} \right) \mathcal{V}_1 + (-i e_1 + e_2) \mathcal{V}_2 &= 0, \\
- (i e_1 + e_2) \mathcal{V}_1 + (-i e_3 +4) \mathcal{V}_2 &= 0.
\end{split}
\end{align}
By using the notation in Appendix \ref{sec:R irr rep}, 
elements of $\ker D$ are explicitly described as 
\begin{align} \label{eq: elements of ker D}
\begin{split}
\mathcal{V}_{1} &= 
-i \sqrt{\frac{7}{10}} \langle \rho_{6} (\cdot) v^{(6)}_{5}, u_{1} \rangle 
- 2i  
\sqrt{\frac{7}{6}}
\langle \rho_{4} (\cdot) v^{(4)}_{3}, u_{2} \rangle, \\
\mathcal{V}_{2} &= 
\langle \rho_{6} (\cdot) v^{(6)}_{4}, u_{1} \rangle 
+ \langle \rho_{4} (\cdot) v^{(4)}_{2}, u_{2} \rangle
+ \langle \rho_{4} (\cdot) v^{(4)}_{4}, u_{3} \rangle
\end{split}
\end{align}
for $u_{1} \in V_{6}, u_{2}, u_{3} \in V_{4}$. 
\end{lemma}

\begin{lemma} \label{lem:L2 A3}
For $V, W \in \ker D$, 
the $L^2$ inner product of 
$
\left. \frac{d^2}{d t^2} F(t V) \right|_{t=0}
$
and $W$ is given by 
$$
\left \langle \left. \frac{d^2}{d t^2} F(t V) \right|_{t=0}, W \right \rangle_{L^2} 
=
\frac{4 \sqrt{3}}{7} {\rm Re} \left( I(V,W) + I(V+W,V) - I(V,V) - I(W,V) \right). 
$$
Here, 
$$
I(V,W) = \int_{{\rm SU}(2)} 
\left( \mathcal{V}_1 \mathcal{V}_2 \cdot \overline{(-i e_1 + e_2) \mathcal{W}_1}
+ \frac{1}{2} \mathcal{V}_2^2 \cdot \overline{(3i e_3 -8) \mathcal{W}_1}
\right) dg, 
$$
where 
$V = \sum_{j=1}^4 V_j \eta_j, W = \sum_{j=1}^4 W_j \eta_j$, 
$\mathcal{V}_1 = V_1 + i V_2, \mathcal{V}_2 = V_3 - i V_4$, 
$\mathcal{W}_1 = W_1 + i W_2$ and $\mathcal{W}_2 = W_3 - i W_4$. 
\end{lemma}

\begin{proof}
Use the notation in Lemma \ref{lem:explicit second der A3}. 
First note that 
\begin{align*}
\left \langle \left. \frac{d^2}{d t^2} F(t V_1) \right|_{t=0}, W \right \rangle_{L^2} 
=&
{\rm Re} 
\int_{{\rm SU}(2)} 
\left(
(F_1 + i F_2) \cdot \bar{\mathcal{W}}_1 + (F_3 -i F_4) \cdot \bar{\mathcal{W}}_2
\right) dg.
\end{align*}
By using Lemma \ref{lem:integ by parts}, 
we can integrate by parts to obtain 
\begin{align*}
- \int_{{\rm SU}(2)}
(i e_1 + e_2) (\mathcal{V}_1 \mathcal{V}_2) \cdot \bar{\mathcal{W}}_1 dg
=
\int_{{\rm SU}(2)}
\mathcal{V}_1 \mathcal{V}_2 \cdot \overline{(- i e_1 + e_2) \mathcal{W}_1} dg, 
\end{align*}
\begin{align*}
&\int_{{\rm SU}(2)}
\left(
\left(i e_3 - \frac{24}{7} \right) (\mathcal{V}_2^2) \cdot \bar{\mathcal{W}}_1
+
\frac{1}{2} (i e_1 + e_2) (\mathcal{V}_2^2) \cdot \bar{\mathcal{W}}_2 \right) dg \\
=&
\int_{{\rm SU}(2)}
\mathcal{V}_2^2 \cdot
\overline{
\left \{
\left(i e_3 - \frac{24}{7} \right) \mathcal{W}_1 
- \frac{1}{2} (-i e_1 + e_2) \mathcal{W}_2 \right \} 
} dg \\
\stackrel{(\ref{eq:DV=0})}
=&
\frac{1}{2}
\int_{{\rm SU}(2)}
\mathcal{V}_2^2 \cdot
\overline{
(3 i e_3 -8) \mathcal{W}_1} dg.
\end{align*}
We also have 
$$
\bar{\mathcal{V}}_2 \left( \left( 2 i e_3 - \frac{48}{7} \right)\mathcal{V}_1 
-(-i e_1 + e_2) \mathcal{V}_2  \right) 
\stackrel{(\ref{eq:DV=0})}
=
\bar{\mathcal{V}}_2 \cdot 
(3 i e_3 - 8) \mathcal{V}_1.
$$
Thus it follows that 
\begin{align*}
&\left \langle \left. \frac{d^2}{d t^2} F(t V) \right|_{t=0}, W \right \rangle_{L^2} \\
=&
\frac{4 \sqrt{3}}{7} 
{\rm Re} 
\int_{{\rm SU}(2)} 
\left(
\mathcal{V}_1 \mathcal{V}_2 \cdot \overline{(-i e_1 + e_2) \mathcal{W}_1}
+ 
\frac{1}{2} \mathcal{V}_2^2 \cdot \overline{(3 i e_3-8) \mathcal{W}_1} \right. \\
&
\qquad \qquad 
\left.
+
(\bar{\mathcal{V}}_2 \bar{\mathcal{W}}_1 + \bar{\mathcal{V}}_1 \bar{\mathcal{W}}_2) 
\cdot (-i e_1 + e_2) \mathcal{V}_1
+
\bar{\mathcal{V}}_2 \bar{\mathcal{W}}_2
\cdot (3 i e_3-8) \mathcal{V}_1
\right) dg.
\end{align*}
From the equations 
$
\mathcal{V}_2 \mathcal{W}_1 + \mathcal{V}_1 \mathcal{W}_2
=
(\mathcal{V}_1 + \mathcal{W}_1) (\mathcal{V}_2 + \mathcal{W}_2) 
- (\mathcal{V}_1 \mathcal{V}_2 + \mathcal{W}_1 \mathcal{W}_2)
$
and 
$
2 \mathcal{V}_2 \mathcal{W}_2 
= (\mathcal{V}_2 + \mathcal{W}_2)^2 - \mathcal{V}_2^2 - \mathcal{W}_2^2, 
$
the proof is done.
\end{proof}

Thus we only have to calculate 
$I (V, W)$ for any $V, W \in \ker D$ 
to compute 
$\left \langle \left. \frac{d^2}{d t^2} F(t V) \right|_{t=0}, W \right \rangle_{L^2}. 
$
In fact, we have the following.

\begin{lemma} \label{lem:I(V,W) A3}
For $V, W \in \ker D$, we have 
$$
I(V,W) = 0.
$$
\end{lemma}

\begin{proof}
For 
$V = \sum_{j=1}^4 V_j \eta_j$ and $W = \sum_{j=1}^4 W_j \eta_j$, 
set 
$\mathcal{V}_1 = V_1 + i V_2, \mathcal{V}_2 = V_3 - i V_4$, 
$\mathcal{W}_1 = W_1 + i W_2$ and $\mathcal{W}_2 = W_3 - i W_4$. 
By Lemma \ref{lem:DV=0}, we may assume that 
$
\mathcal{V}_{1}, \mathcal{V}_{2}$ 
are given by (\ref{eq: elements of ker D}) for 
$u_{1} \in V_{6}, u_{2}, u_{3} \in V_{4}$ 
and 
$
\mathcal{W}_{1}, \mathcal{W}_{2}$ 
are given by the right hand side of (\ref{eq: elements of ker D}), 
where we replace $u_j$ 
with $w_j$ for $j=1,2,3$ and 
$w_{1} \in V_{6}, w_{2}, w_{3} \in V_{4}$.

By (\ref{eq:diff by Ei}) 
and $\{ e_1, e_2, e_3\} = \{ E_1/\sqrt{7}, E_2/\sqrt{7}, E_3 \}$, 
note that 
\begin{align*}
(-i e_1 +e_2) \mathcal{W}_1
&=
2 \sqrt{\frac{3}{5}} \langle \rho_{6} (\cdot) v^{(6)}_{6}, w_{1} \rangle
+ 
\frac{8}{\sqrt{6}} \langle \rho_{4} (\cdot) v^{(4)}_{4}, w_{2} \rangle, \\
(3 i e_3 - 8) \mathcal{W}_1 
&= 
-4i \sqrt{\frac{7}{10}} \langle \rho_{6} (\cdot) v^{(6)}_{5}, w_{1} \rangle
+ 
4i \sqrt{\frac{7}{6}} \langle \rho_{4} (\cdot) v^{(4)}_{3}, w_{2} \rangle. 
\end{align*}
Then by Lemmas \ref{lem:0 integ abc} and \ref{lem:0 integ 4,6}, we compute 
\begin{align*}
I(V,W) =& \int_{{\rm SU}(2)} 
\left( \mathcal{V}_1 \mathcal{V}_2 \cdot \overline{(-i e_1 + e_2) \mathcal{W}_1}
+ \frac{1}{2} \mathcal{V}_2^2 \cdot \overline{(3i e_3 -8) \mathcal{W}_1}
\right) dg \\
=&
\int_{{\rm SU}(2)} 
\left(
-2i \sqrt{\frac{7}{6}} \langle \rho_{4} (g) v^{(4)}_{3}, u_{2} \rangle
\langle \rho_{4} (g) v^{(4)}_{4}, u_{3} \rangle
\cdot 
2 \sqrt{\frac{3}{5}} \overline{\langle \rho_{6} (g) v^{(6)}_{6}, w_{1} \rangle}  \right. \\
&\qquad -
i \sqrt{\frac{7}{10}} \langle \rho_{6} (g) v^{(6)}_{5}, u_{1} \rangle
\langle \rho_{4} (g) v^{(4)}_{2}, u_{2} \rangle
\cdot 
\frac{8}{\sqrt{6}} \overline{\langle \rho_{4} (g) v^{(4)}_{4}, w_{2} \rangle} \\
&\qquad -
2i \sqrt{\frac{7}{6}} \langle \rho_{4} (g) v^{(4)}_{3}, u_{2} \rangle
\langle \rho_{6} (g) v^{(6)}_{4}, u_{1} \rangle
\cdot 
\frac{8}{\sqrt{6}} \overline{\langle \rho_{4} (g) v^{(4)}_{4}, w_{2} \rangle} \\
&\qquad +
4i \sqrt{\frac{7}{10}} \langle \rho_{4} (g) v^{(4)}_{2}, u_{2} \rangle
\langle \rho_{4} (g) v^{(4)}_{4}, u_{3} \rangle
\cdot 
\overline{\langle \rho_{6} (g) v^{(6)}_{5}, w_{1} \rangle} \\
&\qquad \left. -
4i \sqrt{\frac{7}{6}} \langle \rho_{6} (g) v^{(6)}_{4}, u_{1} \rangle
\langle \rho_{4} (g) v^{(4)}_{2}, u_{2} \rangle
\cdot 
\overline{\langle \rho_{4} (g) v^{(4)}_{3}, w_{2} \rangle} 
\right) dg\\
\stackrel{(\ref{eq:conjugate})}
=&
\int_{{\rm SU}(2)} 
\left(
-4i \sqrt{\frac{7}{10}} \langle \rho_{4} (g) v^{(4)}_{3}, u_{2} \rangle
\langle \rho_{4} (g) v^{(4)}_{4}, u_{3} \rangle
\cdot 
\overline{\langle \rho_{6} (g) v^{(6)}_{6}, w_{1} \rangle}  \right. \\
&\qquad +
i \sqrt{\frac{7}{10}} \cdot 
\frac{8}{\sqrt{6}}
\langle \rho_{4} (g) v^{(4)}_{2}, u_{2} \rangle 
\langle \rho_{4} (g) v^{(4)}_{0}, w_{2}^* \rangle
\cdot  \overline{
\langle \rho_{6} (g) v^{(6)}_{1}, u_{1}^* \rangle} \\
&\qquad -
2i \sqrt{\frac{7}{6}} \cdot \frac{8}{\sqrt{6}} 
\langle \rho_{4} (g) v^{(4)}_{3}, u_{2} \rangle
\langle \rho_{4} (g) v^{(4)}_{0}, w_{2}^* \rangle
\cdot 
\overline{
\langle \rho_{6} (g) v^{(6)}_{2}, u_{1}^* \rangle} \\
&\qquad +
4i \sqrt{\frac{7}{10}} \langle \rho_{4} (g) v^{(4)}_{2}, u_{2} \rangle
\langle \rho_{4} (g) v^{(4)}_{4}, u_{3} \rangle
\cdot 
\overline{\langle \rho_{6} (g) v^{(6)}_{5}, w_{1} \rangle} \\
&\qquad \left. 
+
4i \sqrt{\frac{7}{6}} 
\langle \rho_{4} (g) v^{(4)}_{2}, u_{2} \rangle
\langle \rho_{4} (g) v^{(4)}_{1}, w_{2}^* \rangle
\cdot 
\overline{
\langle \rho_{6} (g) v^{(6)}_{2}, u_{1}^* \rangle
} 
\right) dg.
\end{align*}
By Lemmas \ref{lem:Schur orthog sym2} 
and \ref{lem:explicit alpha}, we further compute 
\begin{align*}
=&
\frac{1}{7}
\left(
- 4i \sqrt{\frac{7}{10}}  
\langle  v^{(4)}_{3} \otimes v^{(4)}_{4}, 
\alpha_{4,4,1} (v^{(6)}_{6}) \rangle 
\cdot
\overline{
\langle 
u_{2} \otimes u_{3}, \alpha_{4,4,1} (w_{1})
\rangle}
\right. \\
&\qquad +
i \sqrt{\frac{7}{10}} \cdot \frac{8}{\sqrt{6}}  
\langle  v^{(4)}_{2} \otimes v^{(4)}_{0}, 
\alpha_{4,4,1} (v^{(6)}_{1}) \rangle 
\cdot
\overline{
\langle 
u_{2} \otimes w_{2}^*, \alpha_{4,4,1} (u_{1}^*)
\rangle}\\
&\qquad -
2i \sqrt{\frac{7}{6}} \cdot \frac{8}{\sqrt{6}} 
\langle  v^{(4)}_{3} \otimes v^{(4)}_{0}, 
\alpha_{4,4,1} (v^{(6)}_{2}) \rangle 
\cdot
\overline{
\langle 
u_{2} \otimes w_{2}^*, \alpha_{4,4,1} (u_{1}^*)
\rangle}\\
&\qquad + 
4i \sqrt{\frac{7}{10}}
\langle  v^{(4)}_{2} \otimes v^{(4)}_{4}, 
\alpha_{4,4,1} (v^{(6)}_{5}) \rangle 
\cdot
\overline{
\langle 
u_{2} \otimes u_{3}, \alpha_{4,4,1} (w_{1})
\rangle}\\
&\qquad \left. 
+
4i \sqrt{\frac{7}{6}} 
\langle  v^{(4)}_{2} \otimes v^{(4)}_{1}, 
\alpha_{4,4,1} (v^{(6)}_{2}) \rangle 
\cdot
\overline{
\langle 
u_{2} \otimes w_{2}^*, \alpha_{4,4,1} (u_{1}^*)
\rangle} \right) \\
\end{align*}
\begin{align*}
=&
\frac{1}{7}
\left(
4i \sqrt{\frac{7}{10}} \cdot \sqrt{c_{4,4,1}} (-24 \sqrt{5} + 24 \sqrt{5}) \cdot 
\overline{
\langle 
u_{2} \otimes u_{3}, \alpha_{4,4,1} (w_{1})
\rangle}
\right. \\
&\left. 
+
4i \sqrt{\frac{7}{6}} \cdot \sqrt{c_{4,4,1}}
\left(
\frac{2}{\sqrt{10}} \cdot (-24 \sqrt{5}) + \frac{4}{\sqrt{6}} \cdot 24 \sqrt{3} -24 \sqrt{2}
\right)
\cdot 
\overline{
\langle 
u_{2} \otimes w_{2}^*, \alpha_{4,4,1} (u_{1}^*)
\rangle} \right) \\
=&0.
\end{align*}
\end{proof}

Theorem \ref{thm:A3 unob to second order} follows from these lemmas.

\begin{proof}[Proof of Theorem \ref{thm:A3 unob to second order}]
Recall that $D$ given in 
Proposition \ref{prop:explicit D} is self-adjoint 
by Lemma \ref{lem:self-ad}. 
Then 
by Lemma \ref{lem:unob second order}, 
we only have to show that 
$
\left \langle \left. \frac{d^2}{d t^2} F(t V) \right|_{t=0}, W \right \rangle_{L^2} = 0
$
for any $V, W \in \ker D$. 
This equation is satisfied 
by Lemmas \ref{lem:L2 A3} and \ref{lem:I(V,W) A3}.  
\end{proof}

\subsection{Deformations of $A_3$ arising from ${\rm Spin}(7)$}

To see whether 
infinitesimal associative deformations of $A_3$ extend to actual deformations, 
it would be important to understand the 
trivial deformations (deformations given by the ${\rm Spin}(7)$-action) 
of $A_3$. 
Since $A_3 \cong {\rm SU}(2)$, 
the dimension of the subgroup of ${\rm Spin}(7)$ preserving $A_3$ is at least 3. 
We show that it is 4-dimensional. 
More precisely, we have the following.

\begin{lemma}\label{lem:stab A3}
Use the notation in (\ref{eq: irr SU(2) on C4}), (\ref{eq:def of eta i}), 
Lemmas \ref{lem:irr decomp su(4)} and \ref{lem:decomp spin(7) su(4)}.
Set $p_0 = \frac{1}{\sqrt{2}} {}^t\! (0, 1, i, 0)$. 
Then we have
\begin{align*}
\left\{ X \in \mathfrak{spin}(7); 
\begin{array}{l}
\langle X \cdot \rho_3 (g) \cdot p_0, 
(\eta_i)_{\rho_3 (g) \cdot p_0} \rangle = 0 \\
\mbox{ for any } g \in {\rm SU}(2) 
\mbox{ and } i= 1, \cdots, 4
\end{array}
\right\}
=
W_1^{\mathfrak{spin}(7)}
\oplus 
W_3^{\mathfrak{su}(4)}.
\end{align*}
\end{lemma}
\begin{proof}
Since the left hand side is ${\rm SU}(2)$-invariant, 
it is a direct sum of $W_k^{\mathfrak{spin}(7)}$'s or $W_l^{\mathfrak{su}(4)}$'s. 
Thus we only have to see 
whether 
an element in $W_k^{\mathfrak{spin}(7)}$ or $W_l^{\mathfrak{su}(4)}$ 
is contained in the left hand side. 

By definition, $W_3^{\mathfrak{su}(4)}$ is contained in the left hand side. 
Via the identification of $\mathbb{C}^4 \cong \mathbb{R}^8$ given by (\ref{eq:identification R8 C4}), 
we see that 
$$
(\rho_3 (g^{-1}) H_0 \rho_3 (g)) \cdot p_{0}
=
H_0 \cdot p_{0}
=
\frac{1}{\sqrt{2}} {}^t\! (0,0,0,1,1,0,0,0)
=
(\rho_3)_* (E_3) \cdot p_{0}
$$
for any $g \in {\rm SU}(2).$
Hence 
$W_1^{\mathfrak{spin}(7)}$ is contained in the left hand side.

For 
$X = 
\left( 
\begin{array}{cccc}
0    & i    &0    & 0\\
i     & 0   & 0   & 0 \\
0    & 0   &0    & -i\\
0    & 0   & -i  &0 
\end{array} 
\right) 
\in 
W_5^{\mathfrak{su}(4)}, 
Y =
\left( 
\begin{array}{cccc}
0   & 0  & 1  & 0 \\
0   & 0  & 0  & -1\\
-1 & 0  & 0  & 0 \\
0   & 1  &0   & 0
\end{array} 
\right) 
\in 
W_7^{\mathfrak{su}(4)}, 
$ and 
$
Z =  
\left( 
\begin{array}{cccc}
0   & 0  & -1 & 0 \\
0   & 0  & 0  & 1\\
1   & 0  & 0  & 0 \\
0   & -1 &0   & 0
\end{array} 
\right) 
\oplus 
\left( 
\begin{array}{cccc}
0   & 0  & 1  & 0 \\
0   & 0  & 0  & -1\\
-1 & 0  & 0  & 0 \\
0   & 1  &0   & 0
\end{array} 
\right) 
\in 
W_5^{\mathfrak{spin}(7)}
$, 
we have 
$
\langle X \cdot p_0, (\eta_1)_{p_0} \rangle 
= 1, 
\langle Y \cdot p_0, (\eta_1)_{p_0} \rangle 
= 1$ 
and 
$\langle Z \cdot p_0, (\eta_4)_{p_0} \rangle 
= 2/\sqrt{7}$. 
Note that 
via the identification of $\mathbb{C}^4 \cong \mathbb{R}^8$ given by (\ref{eq:identification R8 C4}) 
$$
p_{0}
=\frac{1}{\sqrt{2}} {}^t\! (0,0,1,0,0,1,0,0) 
\qquad \mbox{and} \qquad 
(\eta_4)_{p_0}
=\frac{1}{\sqrt{42}} {}^t\! (-2 \sqrt{3},0,0,3,-3,0,0,2 \sqrt{3}).
$$
Hence 
$W_5^{\mathfrak{su}(4)}$, 
$W_7^{\mathfrak{su}(4)}$ and 
$W_5^{\mathfrak{spin}(7)}$
are not contained in the left hand side.
\end{proof}

Hence 
by Lemma \ref{lem:stab A3}, we see that 
the space of trivial deformations of $A_3$ is 
isomorphic to 
\begin{align*}
\mathfrak{spin}(7)/ (W_1^{\mathfrak{spin}(7)}
\oplus W_3^{\mathfrak{su}(4)})
\cong 
W_5^{\mathfrak{spin}(7)}
\oplus 
W_5^{\mathfrak{su}(4)}
\oplus 
W_7^{\mathfrak{su}(4)},
\end{align*}
which is a 17-dimensional subspace of 
the 34-dimensional space $\ker D$.

\begin{remark}
Use the notation in (\ref{eq:def of j}), 
Lemmas \ref{lem:diff mat coefficient}, \ref{lem:irr decomp su(4)} and \ref{lem:decomp spin(7) su(4)}.
By tedious calculations, 
we can describe 
elements of $\ker D$ 
given by $\mathfrak{spin}(7)/ (W_1^{\mathfrak{spin}(7)}
\oplus W_3^{\mathfrak{su}(4)})
\cong 
W_5^{\mathfrak{spin}(7)}
\oplus 
W_5^{\mathfrak{su}(4)}
\oplus 
W_7^{\mathfrak{su}(4)}
$.  
Elements in $\ker D$ are of the form (\ref{eq: elements of ker D}). 
In the following table, 
each space in the left-hand side  
corresponds to the 
elements in $\ker D$ given 
by the right-hand side. \\

\begin{tabular}{|c||c|c|c|c|} 
\hline
                                     & $\ker D$ \\ \hline \hline 
$W_7^{\mathfrak{su}(4)}$   & $u_1 \in (1-j)V_6, \quad u_2=u_3=0.$ \\ \hline
$W_5^{\mathfrak{su}(4)}$   & $u_1=0, \quad u_2 \in (1-j) V_4, \quad 
                                         u_3 = (2 \sqrt{6}/3) \cdot u_2^*$. \\ \hline
$W_5^{\mathfrak{spin}(7)}$ & $u_1=0,  \quad u_2 \in (1+j) V_4, \quad 
                                          u_3 = (2 \sqrt{6}/3) \cdot u_2^*$. \\ \hline
\end{tabular}

\end{remark}

\section*{Appendix}
\appendix

\section{Representations of ${\rm SU}(2)$} \label{sec:R irr rep}

In this section, we summarize the results 
about representations of ${\rm SU}(2)$. 
First, we recall the $\mathbb{C}$-irreducible representations of ${\rm SU}(2)$.

Let 
$V_{n}$ be a $\mathbb{C}$-vector space 
of all complex homogeneous polynomials 
with two variables $z_{1}, z_{2}$ of degree $n$, where $n \geq 0$, 
and 
define the representation 
$\rho_{n} : {\rm SU}(2) \rightarrow {\rm GL}(V_{n})$ as 
\begin{align*} 
\left(\rho_{n} 
\left( 
\begin{array}{cc}
a & -\overline{b} \\
b &   \overline{a} \\
\end{array} 
\right) f
\right) (z_{1}, z_{2}) 
=
f
\left( 
(z_{1}, z_{2}) 
\left(
\begin{array}{cc}
a & -\overline{b} \\
b &   \overline{a} \\
\end{array} 
\right) 
\right).  
\end{align*} 
Define the Hermitian inner product $\langle \ , \ \rangle$ 
of $V_{n}$ such that 
\begin{align*} 
\left \{
v^{(n)}_{k}
=
z_{1}^{n-k} z_{2}^{k} / \sqrt{k ! (n - k)!}
\right \}_{0 \leq k \leq n}  
\end{align*} 
is a unitary basis of $V_{n}$. 
Denoting by $\widehat{{\rm SU}(2)}$  
the set of all equivalence classes of finite dimensional irreducible representations of ${\rm SU}(2)$, 
we know that 
$\widehat{{\rm SU}(2)} = \{ (V_{n} ,\rho_{n}) ; n \geq 0 \}$. 
Then 
every $\mathbb{C}$-valued continuous function 
on ${\rm SU}(2)$ is 
uniformly approximated by 
the $\mathbb{C}$-linear combination of  
\begin{align*}
\left \{ \langle \rho_{n} (\cdot) v^{(n)}_{i}, v^{(n)}_{j} \rangle ; 
n \geq 0, \ 0 \leq i, j \leq n
\right \}, 
\end{align*} 
which are mutually orthogonal with respect to the $L^{2}$ inner product. \\

Next, we review the $\mathbb{R}$-irreducible representations of ${\rm SU}(2)$ 
by \cite[Section 2]{Mashimo minimal}. 
A more general reference of this topic is \cite{Oni}.

Define the map $j: V_n \rightarrow V_n$ by
\begin{align} \label{eq:def of j}
(j f) (z_1, z_2) = \overline{f(-\bar{z}_2, \bar{z}_1)}, 
\end{align}
which is a $\mathbb{C}$-antilinear ${\rm SU}(2)$-equivariant map 
satisfying $j^2 = (-1)^n$.
This map $j$ is called a {\bf structure map} (\cite[Section 2]{Mashimo minimal}).

When $n$ is even, we have $j^2=1$ 
and $V_n$ decomposes into two mutually equivalent 
real irreducible representations: 
$V_n = (1+j) V_n \oplus (1-j) V_n.$ 
When $n$ is odd,
 $V_n$ is also irreducible as a real representation. 

All of the real irreducible representations are given in this way, 
and hence their dimensions are given by $4m$ or $2n+1$ for $m,n \geq 0$.
Denote by $W_k$, where $k \in 4 \mathbb{Z} \cup (2 \mathbb{Z} +1)$, 
the $k$-dimensional $\mathbb{R}$-irreducible representation of ${\rm SU}(2)$. 
It follows that 
\begin{align} \label{eq:decomp C R}
V_{2m+1} = W_{4m+4}, \qquad V_{2m} = W_{2m+1} \oplus W_{2m+1}
\qquad
\mbox{ for } m \geq 0.
\end{align}

The characters $\chi_{V_n}$ of $V_n$ 
are determined by the values on 
the maximal torus 
$\left \{ h_a = 
\left( 
\begin{array}{cc}
a & 0 \\
0 & a^{-1}
\end{array} 
\right);
a \in \mathbb{C}, |a| = 1 \right \}$ of ${\rm SU}(2)$. 
It is well-known that 
\begin{align*}
\chi_{V_n} (h_a)
=
\sum_{k=0}^{n} a^{2k-n} = \frac{a^{n+1} - a^{-(n+1)} }{a-a^{-1}}.
\end{align*}
By (\ref{eq:decomp C R}), 
the characters 
$\chi_{W_k}$
of $W_k$ on the maximal torus 
are given by 
\begin{align}
\begin{split}
\chi_{W_{4m+4}}(h_a) 
&= 
2 \chi_{V_{2m+1}} (h_a)
=
2 \sum_{k=0}^{2m+1} a^{2k-(2m+1)} = \frac{2 (a^{2m+2} - a^{-(2m+2)})}{a-a^{-1}}, \\
\chi_{W_{2m+1}}(h_a) 
&= 
\chi_{V_{2m}} (h_a)
=
\sum_{k=0}^{2m} a^{2k-2m} = \frac{a^{2m+1} - a^{-(2m+1)} }{a-a^{-1}}.
\label{eq:character R irr rep}
\end{split}
\end{align}

Finally, we summarize technical lemmas. 

\begin{lemma}[{\cite[Lemma 6.9]{Kdeform}}]
\label{lem:diff mat coefficient}
For $u = \sum_{l = 0}^{n} C_{l} v^{(n)}_{l} \in V_{n}$, 
set  
\begin{align*}
u^{*} = j u = \sum_{l = 0}^{n} (-1)^{n - l} \overline{C}_{n - l} v^{(n)}_{l} \in V_{n}. 
\end{align*}
Then 
for any $n \geq 0, 0 \leq k \leq n,  u \in V_{n}$, we have 
\begin{align} \label{eq:conjugate}
\overline{ \langle \rho_{n} (\cdot) v^{(n)}_{k}, u \rangle} 
&=
(-1)^{k}
\langle \rho_{n} (\cdot) v^{(n)}_{n - k}, u^{*} \rangle.
\end{align}
Let $\{ E_1, E_2, E_3 \}$ be the basis of 
the Lie algebra $\mathfrak{su}(2)$ of ${\rm SU}(2)$ given by (\ref{E1E2E3}). 
Identify $E_i \in \mathfrak{su}(2)$  
with 
the left invariant differential operator on ${\rm SU}(2)$. 
Then 
\begin{align} \label{eq:diff by Ei}
\begin{split}
(-i E_{1} + E_{2}) 
\langle \rho_{n} (\cdot) v^{(n)}_{k}, u \rangle
&=
\left\{
\begin{array}{ll}
2i \sqrt{(k + 1)(n - k)} \langle \rho_{n} (\cdot) v^{(n)}_{k + 1}, u \rangle, & (k < n) \\
0,                                                                                                      & (k = n)
\end{array}
\right.\\
(i E_{1} + E_{2}) 
\langle \rho_{n} (\cdot) v^{(n)}_{k}, u \rangle
&=
\left\{
\begin{array}{ll}
2i \sqrt{k (n - k + 1)} \langle \rho_{n} (\cdot) v^{(n)}_{k - 1}, u \rangle, & (k > 0) \\
0,                                                                                                      & (k = 0)
\end{array}
\right.\\
i E_{3} \langle \rho_{n} (\cdot) v^{(n)}_{k}, u \rangle
&= (-n + 2k) \langle \rho_{n} (\cdot) v^{(n)}_{k}, u \rangle. \\
\end{split}
\end{align}
\end{lemma}

Since the Haar measure 
is ${\rm SU}(2)$-invariant, 
we have the following.
\begin{lemma} \label{lem:integ by parts}
For any $X \in \mathfrak{su}(2)$ and a smooth function $f$ on ${\rm SU}(2)$, we have 
$$
\int_{{\rm SU}(2)} X(f) (g) dg = 0.
$$
\end{lemma}

\section{Clebsch-Gordan decomposition}

Use the notation in Section \ref{sec:R irr rep}. 
In the computation in Section \ref{sec:computation S7}, 
we need the irreducible decomposition of $V_m \otimes V_n$ for $m,n \geq 0$. 
This is well-known as the Clebsch-Gordan decomposition: 
$$
V_m \otimes V_n = \bigoplus_{h=0}^{{\rm min} \{m,n \}} V_{m+n-2h}.
$$
Identify 
$V_m \otimes V_n$ with 
the vector subspace of polynomials in $(z_1, z_2, w_1, w_2)$ 
consisting of 
homogeneous polynomials of degree $m$ in $(z_1, z_2)$ and of
degree $n$ in $(w_1, w_2)$. 
Then the inclusion 
$V_{m+n-2h} \rightarrow V_m \otimes V_n$ is explicitly given as follows.

\begin{lemma}[{\cite[p.46]{Procesi}, \cite[Section 2.1.2]{Al Nuwairan}}] \label{lem:Clebsch-Gordan}
For $0 \leq h \leq {\rm min} \{m,n \}$, define the map
$$
\alpha_{m,n,h}: V_{m+n-2h} \rightarrow V_m \otimes V_n
$$
by
$$
\alpha_{m,n,h} (f(z_1,z_2)) = \sqrt{c_{m,n,h}} (z_1 w_2 -z_2 w_1)^h 
\left( w_1\frac{\partial}{\partial z_1} + w_2 \frac{\partial}{\partial z_2} \right)^{n-h} (f(z_1,z_2)),
$$
where 
$c_{m,n,h} > 0$ is given in \cite[Section 2.2.2]{Al Nuwairan}.
Then the map $\alpha_{m,n,h}$ is ${\rm SU}(2)$ equivariant and isometric. 
\end{lemma}

Denote by $\rho_{m,n}$ 
the induced representation of ${\rm SU}(2)$ on $V_m \otimes V_n$. 
Since we know that
$$
\langle \rho_m (g) u_m, u'_m \rangle  \langle \rho_n (g) u_n, u'_n \rangle
= 
\langle \rho_{m,n} (g) (u_m \otimes u_n), u'_m \otimes u'_n \rangle
$$
for $u_m, u'_m \in V_m, u_n, u'_n \in V_n$ and $g \in {\rm SU}(2)$, 
we have the following 
by Lemma \ref{lem:Clebsch-Gordan} 
and the Schur orthogonality relations.

\begin{lemma} \label{lem:Schur orthog sym2}
Set $r=m+n-2h$.Then we have 
\begin{align*}
&\int_{{\rm SU}(2)} 
\langle \rho_m (g) u_m, u'_m \rangle  \langle \rho_n (g) u_n, u'_n \rangle 
\overline{\langle \rho_r (g) u_r, u'_r \rangle} dg \\
=&
\frac{1}{r+1} 
\langle u_m \otimes u_n, \alpha_{m,n,h} (u_r) \rangle
\overline{
\langle u'_m \otimes u'_n, \alpha_{m,n,h} (u'_r) \rangle
}
\end{align*}
for $u_j, u'_j \in V_j$.
\end{lemma}

The next lemma is very useful for the computation in Section \ref{sec:computation S7}. 

\begin{lemma} \label{lem:0 integ abc}
\begin{align*}
\int_{{\rm SU}(2)} 
\langle \rho_m (g) v^{(m)}_a, u'_m \rangle  \langle \rho_n (g) v^{(n)}_b, u'_n \rangle 
\overline{\langle \rho_r (g) v^{(r)}_c, u'_r \rangle} dg
= 0
\end{align*}
for any $u'_m \in V_m, u'_n \in V_n, u'_r \in V_r$ if
$$
a+b \neq c+h \left( = c+ \frac{m+n-r}{2} \right). 
$$
\end{lemma}

\begin{proof}
We compute
\begin{align*}
&\left( \sqrt{r! (r-c)!}/ \sqrt{c_{m,n,h}} \right)
\alpha_{m,n,h} (v^{(r)}_c) \\
=&
(z_1 w_2-z_2 w_1)^h 
\left( w_1\frac{\partial}{\partial z_1} + w_2 \frac{\partial}{\partial z_2} \right)^{n-h} (z_1^{r-c} z_2^c)\\
=&
\sum_{i=0}^h \sum_{j=0}^{n-h} 
\left(\begin{array}{c} 
h\\
i 
\end{array}\right)
\left(\begin{array}{c} 
n-h\\
j
\end{array}\right)
(z_1 w_2)^i (- z_2 w_1)^{h-i} 
w_1^j w_2^{n-h-j} \left(\frac{\partial}{\partial z_1} \right )^j 
\left(\frac{\partial}{\partial z_2} \right )^{n-h-j} 
(z_1^{r-c} z_2^c)\\
\in&
{\rm span} \left \{v^{(m)}_{(h-i)+c-(n-h-j)} \otimes v^{(n)}_{i+(n-h-j)}; 0 \leq i \leq h, 0 \leq j \leq n-h 
\right \}\\
\subset& 
{\rm span} \left \{v^{(m)}_d \otimes v^{(n)}_e;  d + e = c+ h \right \}, 
\end{align*}
which gives the proof.
\end{proof}

In this paper, the case of 
$(m,n,h) = (4,4,1)$ or $(6,6,3)$ is important. 
Recall that  the character of 
the induced representation 
on the second symmetric power 
$S^2 (V_n)$ is given by 
$(\chi_{V_n} (g)^2 + \chi_{V_n} (g^2) )/2$. 
(For example, see \cite[Exercise 2.2]{Fulton Harris}.)
By computing the character
of $S^2 (V_4)$ and $S^2 (V_6)$, 
we see that 
$$
S^2 (V_4) = V_8 \oplus V_4 \oplus V_0, \qquad
S^2 (V_6) = V_{12} \oplus V_8 \oplus V_4 \oplus V_0.$$  
Thus we have
$\alpha_{4,4,1} (V_6) \subset \Lambda^2 V_4$,  
$\alpha_{6,6,3} (V_6) \subset \Lambda^2 V_6$
and we obtain the following.

\begin{lemma} \label{lem:0 integ 4,6}
Suppose that $m=4 \mbox{ or } 6$ 
and $u_m, \hat{u}_m, u'_m, \hat{u}'_m \in V_m$.
If $u_m = \hat{u}_m$ or $u'_m = \hat{u}'_m$, we have 
\begin{align*}
\int_{{\rm SU}(2)} 
\langle \rho_m (g) u_m, u'_m \rangle  
\langle \rho_m (g) \hat{u}_m, \hat{u}'_m \rangle 
\overline{\langle \rho_6 (g) v_6, v'_6 \rangle} dg
= 0
\end{align*}
for any $v_6, v'_6 \in V_6$.
\end{lemma}

The next lemma is straightforward and we omit the proof.
\begin{lemma} \label{lem:explicit alpha}
\begin{align*}
\alpha_{4,4,1} (v^{(6)}_0) &= \sqrt{c_{4,4,1}} \cdot 24 \sqrt{5} v^{(4)}_0 \wedge v^{(4)}_1, \\
\alpha_{4,4,1} (v^{(6)}_1) &= \sqrt{c_{4,4,1}} \cdot 24 \sqrt{5} v^{(4)}_0 \wedge v^{(4)}_2, \\
\alpha_{4,4,1} (v^{(6)}_2) &= \sqrt{c_{4,4,1}} \cdot 24 
                                    (\sqrt{3} v^{(4)}_0 \wedge v^{(4)}_3 + \sqrt{2} v^{(4)}_1 \wedge v^{(4)}_2), \\
\alpha_{4,4,1} (v^{(6)}_3) &= \sqrt{c_{4,4,1}} \cdot 24 
                                    (v^{(4)}_0 \wedge v^{(4)}_4 + 2 v^{(4)}_1 \wedge v^{(4)}_3), \\
\alpha_{4,4,1} (v^{(6)}_4) &= \sqrt{c_{4,4,1}} \cdot 24 
                                    (\sqrt{3} v^{(4)}_1 \wedge v^{(4)}_3 + \sqrt{2} v^{(4)}_2 \wedge v^{(4)}_3), \\
\alpha_{4,4,1} (v^{(6)}_5) &= \sqrt{c_{4,4,1}} \cdot 24 \sqrt{5} v^{(4)}_2 \wedge v^{(4)}_4, \\
\alpha_{4,4,1} (v^{(6)}_6) &= \sqrt{c_{4,4,1}} \cdot 24 \sqrt{5} v^{(4)}_3 \wedge v^{(4)}_4. 
\end{align*}
\end{lemma}

\section{Irreducible decomposition of $\mathfrak{spin}(7)$}

In this section, 
we give an irreducible decomposition of 
the Lie algebra $\mathfrak{spin}(7)$ of ${\rm Spin}(7)$ 
under an ${\rm SU}(2)$-action. 
First, we study the $\mathfrak{su}(4) \subset \mathfrak{spin}(7)$ case.

\begin{lemma} \label{lem:irr decomp su(4)}
Use the notation in Appendix \ref{sec:R irr rep}.
Let ${\rm SU}(2)$ act on $\mathfrak{su}(4)$ 
by the composition of 
$\rho_3: {\rm SU}(2) \hookrightarrow {\rm SU}(4)$ 
given by (\ref{eq: irr SU(2) on C4}) 
and the adjoint action of ${\rm SU}(4)$ on $\mathfrak{su}(4)$. 
Then we have 
$$\mathfrak{su}(4) \cong W_3 \oplus W_5 \oplus W_7.$$
More explicitly, 
$W_k$ corresponds to the $k$-dimensional ${\rm SU}(2)$-invariant subspace 
$W_k^{\mathfrak{su}(4)}$ of  $\mathfrak{su}(4)$, 
where $k=3,5,7$, given by 
\begin{align*}
W_3^{\mathfrak{su}(4)}
&=
(\rho_3)_* \mathfrak{su}(2) 
=
\left \{
\left( 
\begin{array}{cccc}
3i a                  & \sqrt{3} z  &0                      & 0\\
-\sqrt{3} \bar{z} & i a           & 2 z                  & 0 \\
0                      &-2  \bar{z} &-i a                  & \sqrt{3} z\\
0                      &0              &-\sqrt{3} \bar{z} & -3 ia
\end{array} 
\right);
z \in \mathbb{C}, 
a \in \mathbb{R}
\right \}, \\
W_5^{\mathfrak{su}(4)}
&=
\left \{
\left( 
\begin{array}{cccc}
i a                    &  z             &w        & 0\\
-\bar{z}            & -i a           & 0       & w \\
-\bar{w}           &0               &-i a     & -z\\
0                      &-\bar{w}    &\bar{z} & ia
\end{array} 
\right);
z,w \in \mathbb{C}, 
a \in \mathbb{R}
\right \}, \\
W_7^{\mathfrak{su}(4)}
&=
\left \{
\left( 
\begin{array}{cccc}
i a             &  z_1                    & z_2                & z_3\\
-\bar{z}_1   & -3i a                  & - \sqrt{3} z_1  & - z_2 \\
-\bar{z}_2   & \sqrt{3} \bar{z}_1 & 3i a                & z_1 \\
-\bar{z}_3   &\bar{z}_2             &-\bar{z}_1        & -ia
\end{array} 
\right);
z_1, z_2, z_3 \in \mathbb{C}, 
a \in \mathbb{R}
\right \}.
\end{align*}
\end{lemma}
\begin{proof}

First, we compute 
the character of this representation 
on the maximal torus 
by using 
(\ref{eq: irr SU(2) on C4}) and  (\ref{eq:character R irr rep}). 
Then 
we see that it is given by 
$\chi_{W_3} + \chi_{W_5} + \chi_{W_7}$. 
This is a straightforward computation, so we omit it. 
Hence we obtain the first statement.

We easily see that 
the three spaces above are invariant 
by the adjoint action of $(\rho_3)_* \mathfrak{su}(2)$. 
Hence 
the three spaces above are 
${\rm SU}(2)$-invariant 
and the proof is done.
\end{proof}

The explicit description of 
$\mathfrak{spin}(7)$ is given in \cite[Proposition 4.2]{Lotay3}. 
It is straightforward to deduce the following 
so we omit the proof.

\begin{lemma} \label{lem:decomp spin(7) su(4)}
We have
$$
\mathfrak{spin}(7) = 
\mathfrak{su}(4) \oplus 
W_1^{\mathfrak{spin}(7)} \oplus 
W_5^{\mathfrak{spin}(7)}, 
$$
where 
\begin{align*}
W_1^{\mathfrak{spin}(7)} &= \mathbb{R} H_0, 
\qquad 
H_0 = 
\left( 
\begin{array}{cccccccc}
0 & 0 &0 &0 &0 &0 &1 &0 \\
0 & 0 &0 &0 &0 &0 &0 &-1 \\
0 & 0 &0 &0 &-1 &0 &0 &0 \\
0 & 0 &0 &0 &0 &1 &0 &0 \\
0 & 0 &1 &0 &0 &0 &0 &0 \\
0 & 0 &0 &-1 &0 &0 &0 &0 \\
-1 & 0 &0 &0 &0 &0 &0 &0 \\
0   &1 &0 &0 &0 &0 &0 &0 \\
\end{array} 
\right), \\
W_5^{\mathfrak{spin}(7)} 
=&
\left \{
\left( 
\begin{array}{cccccccc}
0 & 0 &-a_1 &-a_2 &-a_3 &-a_4 &0 &-a_5 \\
0 & 0 &-a_2 &a_1 &-a_4 &a_3 &-a_5 &0 \\
a_1 & a_2 &0 &0 &0 &-a_5 &-a_3 &a_4 \\
a_2 & -a_1 &0 &0 &-a_5 &0 &a_4 &a_3 \\
a_3 & a_4 &0 &a_5 &0 &0 &a_1 &-a_2 \\
a_4 & -a_3 &a_5 &0 &0 &0 &-a_2 &-a_1 \\
0 & a_5 &a_3 &-a_4 &-a_1 &a_2 &0 &0 \\
a_5 & 0 &-a_4 &-a_3 &a_2 &a_1 &0 &0 \\
\end{array} 
\right); 
a_1, \cdots. a_5 \in \mathbb{R}
\right \}.
\end{align*}
By using the notation in Appendix \ref{sec:R irr rep}, 
$H_0$ is the structure map $j: V_3 \rightarrow V_3$ given in (\ref{eq:def of j}) 
with respect to the basis $\{ v^{(3)}_0, i v^{(3)}_0, \cdots, v^{(3)}_3, i v^{(3)}_3 \}$. 
\end{lemma}

\begin{lemma}
Use the notation in Appendix \ref{sec:R irr rep}.
Let ${\rm SU}(2)$ act on $\mathfrak{spin}(7)$ 
by the composition of 
$\rho_3: {\rm SU}(2) \hookrightarrow {\rm SU}(4) \subset {\rm Spin}(7)$ 
given by (\ref{eq: irr SU(2) on C4}) 
and the adjoint action of ${\rm Spin}(7)$ on $\mathfrak{spin}(7)$. 
Then we have 
$$\mathfrak{spin}(7) \cong (W_1 \oplus W_5) \oplus (W_3 \oplus W_5 \oplus W_7).$$
The subspaces 
$W_1$ and the first $W_5$ correspond 
to $W_1^{\mathfrak{spin}(7)}$ and $W_5^{\mathfrak{spin}(7)}$ in Lemma \ref{lem:decomp spin(7) su(4)}, 
respectively. 
The subspace 
$W_3 \oplus W_5 \oplus W_7$ 
corresponds to $\mathfrak{su}(4)$, 
whose irreducible decomposition is given in Lemma \ref{lem:irr decomp su(4)}.
\end{lemma}

\begin{proof}
By Lemma \ref{lem:decomp spin(7) su(4)}, 
we only have to prove that 
$H_0$ is invariant under the ${\rm SU}(2)$-action 
and 
$W_5^{\mathfrak{spin}(7)}$ is an irreducible 5-dimensional representation of ${\rm SU}(2)$. 

Since $H_0$ is the structure map, 
it is invariant under the ${\rm SU}(2)$-action. 
We easily see that 
$W_5^{\mathfrak{spin}(7)}$ is invariant by the adjoint action of $(\rho_3)_* \mathfrak{su}(2)$. 
Hence it is ${\rm SU}(2)$-invariant. 
As in the proof of Lemma \ref{lem:irr decomp su(4)}, 
we can compute   
the character of the  ${\rm SU}(2)$-representation on $W_5^{\mathfrak{spin}(7)}$ 
and it is equal to $\chi_{W_5}$. 
Hence the proof is done.
\end{proof}

\address{
Gakushuin University, 1-5-1, Mejiro, Toshima,Tokyo, 171-8588, Japan}
{kkawai@math.gakushuin.ac.jp}

\end{document}